\documentclass[a4paper,leqno]{amsart}

\usepackage{amsmath}
\usepackage{amssymb}
\usepackage{amsfonts}
\usepackage{enumerate}
\usepackage{amsthm}
\usepackage{esint}
\usepackage{bbm}
\usepackage{tabularx}
\usepackage{graphicx}
\usepackage{epstopdf}

\newtheorem{theorem}{Theorem}[section]
\newtheorem{lem}[theorem]{Lemma}

\newtheorem{prop}[theorem]{Proposition}

\theoremstyle{definition}
\newtheorem{defi}[theorem]{Definition}
\newtheorem{rem}[theorem]{Remark}
\newtheorem{exam}[theorem]{Example}

\DeclareMathOperator*{\argmin}{arg\,min}

\renewcommand{\eqref}[1]{\textnormal{(\ref{#1})}}

\numberwithin{equation}{section}

\newcommand{\rmi}{\mathrm{i}}

\newcommand{\trinorm}{|\hskip -1pt\|}
\newcommand{\R}{\mathbb{R}}
\newcommand{\N}{\mathbb{N}}

\newcommand{\SSS}{\mathbb{S}}
\newcommand{\C}{\mathbb{C}}
\newcommand{\nn}{\mathcal{N}}

\title[Reconstruction of inhomogeneities]{Discretization and regularization for the reconstruction of inhomogeneities by scattering measurements}

\author[D.~Di Donato]{Daniela Di Donato}
\author[L.~Rondi]{Luca Rondi}

\address[D.~Di Donato]{Dipartimento di Matematica, Universit\`a degli Studi di Pavia, Italy}
\email{daniela.didonato@unipv.it}

\address[L.~Rondi]{Dipartimento di Matematica, Universit\`a degli Studi di Pavia, Italy}
\email{luca.rondi@unipv.it}

\date{}

\begin{document}

\begin{abstract}
We consider the inverse problem of reconstructing inhomogeneities by performing a finite number of scattering measurements of acoustic type in the
time-harmonic setting. We set up the reconstruction as a fully discrete variational problem with regularization. Such a problem depends on a variety of parameters, that is, the number of measurements, the regularization parameter and the discretization parameter, namely the size of the mesh on which we discretize the unknown coefficients of the Helmholtz type equation modelling our physical system.
We show, through a convergence analysis, that one can carefully choose these parameters in such a way that the solution to this discrete regularized minimum problem is a good approximation of the looked-for solution to the inverse problem.

\bigskip


\medskip

\end{abstract}

\maketitle

\setcounter{section}{0}
\setcounter{secnumdepth}{2}

\section{Introduction}
We consider the following inverse acoustic scattering problem. We aim at reconstructing an unknown inhomogeneity, or penetrable obstacle, by performing acoustic scattering measurements.
Namely, we send one or more time-harmonic acoustic incident waves which are scattered by the presence of the inhomogeneity and we measure the so-called far-field of the corresponding scattered waves. As incident waves we restrict our choice to planar waves, that is, to the incident field
\begin{equation}\label{incidentfieldintro}
u^i(x)=u^i(k,d;x)=e^{k\rmi d\cdot x}\quad\text{for any }x\in \R^N,
\end{equation}
where $d\in \SSS^{N-1}$ is the \emph{direction of propagation} and $k>0$ is the \emph{wavenumber}. The direct problem corresponds to find the total field $u$ and the scattered field $u^s$
solvlng
\begin{equation}\label{pbScatteringintro}
\left\{
\begin{array}{ll}
\mathrm{div} (\sigma \nabla u) +k^2 \eta u= 0 &\text{in } \R^N \\
u=u^i + u^s &\text{in } \R^N \\
\displaystyle \lim_{r\to +\infty } r^{\frac{N-1}{2}} \left(\frac{\partial u^s}{\partial r} - kiu^s\right) =0& r=|x|.
\end{array}
\right.
\end{equation} 
The last condition at infinity, the so-called Sommerfeld radiation condition, implies that the scattered field is radiating and that its asymptotic behavior at infinity is characterized
by its \emph{far-field pattern} $u_{\infty}(\sigma,\eta;k,d;\hat{x})$, $\hat{x}\in\mathbb{S}^{N-1}$ being the outgoing direction.
The inhomogeneity is modelled through the coefficients $\sigma$ and $\eta$. We assume that the medium is homogeneous and isotropic outside our penetrable obstacle, which is
contained in a mildly smooth bounded open set $\Omega$ (we do not ask it to be connected, that is, there might be more than one obstacle at the same time). In our case, we consider $\Omega$ to be known.
Inside the penetrable obstacle,
the unknown medium, modelled by the coefficients $\sigma_0$ and $\eta_0$, differs from the background one and may be nonhomogeneous and anisotropic.
We consider some minimal a priori regularity assumptions on the unknown $\sigma_0$ and $\eta_0$ in $\Omega$. Fixed for simplicity the wavenumber $k>0$,
we measure the corresponding far-field pattern $u_{\infty}(\sigma,\eta;k,d_i;\hat{x}_i)$ for a finite number of pairs $(d_i,\hat{x}_i)$, $i=1,\ldots,n$. That is, we perform a finite number of experiments (we send a finite number of incident waves) and we measure the corresponding far-field pattern at a finite number of outgoing directions. We assume that each of these measurements is corrupted by a measurement error bounded by the \emph{noise level} $\varepsilon>0$.
Our collected data, $\mathcal{F}^{\varepsilon}_{meas}=(\mathcal{F}^{\varepsilon}_{meas}(1),\ldots,\mathcal{F}^{\varepsilon}_{meas}(n))$ , is therefore a vector in $\C^n$ and the inverse problem consists in finding (an approximation of) the unknown coefficients $\sigma_0$ and $\eta_0$ in $\Omega$.

Such an inverse problem is linked to sonar techniques, ultrasound tomography and several other applications in science and technology, we refer to \cite{CK98} and \cite{Isak06} to a thorough discussion on this and other related scattering problems. We just recall that the breakthrough uniqueness results for the determination of inhomogeneities by scattering measurements is in \cite{Isak}. Understanding what is the scattering effect of a given inhomogeneity has been the focus of much recent investigation, we refer to the survey \cite{Cak-Vog} and its references for 
a discussion on this significant line of research.

Here we are more interested in the reconstruction issue rather than uniqueness results.
 We set up a variational method of reconstruction to which we add a suitable Tikhonov type regularization $R$, with \emph{regularization coefficient} $a$, to tackle the well-known ill-posedness of the inverse problem. About regularization techniques for inverse problems we refer for instance to \cite{EHN96}.
 Moreover, we consider a discretized version of our minimization problem by restricting the coefficients to a finite dimensional space $Y^h$, using finite elements on a suitable mesh which is characterized by the \emph{mesh size parameter} $h$.
  Namely, we solve
 \begin{equation}\label{minpbmintro}
 \min\left\{\sum_{i=1}^n|u_{\infty}(\sigma,\eta;k,d_i;\hat{x}_i)-\mathcal{F}^{\varepsilon}_{meas}(i)|^2 +aR(\sigma,\eta): (\sigma,\eta)\in Y^h \right\}
 \end{equation}
which is a fully discrete minimization problem. About $R$, we consider a total variation penalization on $\eta$ and, when $N=2$, on $\sigma$ as well. For $N\geq 3$, instead, we consider a Lipschitz constant penalization on $\sigma$.
Actually, since it is much easier to compute, we replace the Lipschitz constant by the essentially equivalent $W^{1,\infty}$-seminorm.
The a priori assumption on $(\sigma_0,\eta_0)$ is just that $R(\sigma_0,\eta_0)$ is bounded. These 
are in some sense minimal regularity requirements. In fact,
when $N\geq 3$, Lipschitz regularity is crucial to have unique continuation for the Helmholtz type equation and, in turn, to guarantee existence and uniqueness of the solution to the direct problem
\eqref{pbScatteringintro}.

We are interested in answering the following question: how to choose, in terms of the noise level $\varepsilon$, the number of measurements, and their location, the regularization coefficient $a$ and the discretization coefficient $h$
so that the solution to our discrete and regularized minimum problem \eqref{minpbmintro} is a good approximation of a solution $(\sigma_{\infty},\eta_{\infty})$ to our inverse problem, that is,
of $(\sigma_{\infty},\eta_{\infty})$ satisfying
$$u_{\infty}(\sigma_{\infty},\eta_{\infty};k,d;\hat{x})=u_{\infty}(\sigma_0,\eta_0;k,d;\hat{x})\quad\text{for any }(d,\hat{x})\in \mathbb{S}^{N-1}\times  \mathbb{S}^{N-1}.$$
We note that we are approximating with a finite number of data the continuous setting where measurements are collected for all incident directions and all outgoing directions.

We show that this is the case when, for suitable parameters $0<\gamma<2$, $0<\alpha<1/2$ and $0<\beta<2$, we choose $a=a(\varepsilon)=\tilde{a}\varepsilon^{\gamma}$, for some positive constant $\tilde{a}$,
$h=h(\varepsilon)=\varepsilon^{2/\alpha}$, and $n$ is of the order of $\varepsilon^{-\beta}$
(with uniform distribution of $d_i$ and $\hat{x}_i$ in an open subset of $\mathbb{S}^{N-1}$), see Theorem~\ref{mainthm} and Example~\ref{dataexample} for a complete and more general statement. It is important to note that the dependence of these quantities on the noise level is always of polynomial type, which is somewhat surprising given the
exponential ill-posedness nature of the inverse problem.
The proof is based on a careful convergence analysis inspired by $\Gamma$-convergence techniques. We recall that the first use of $\Gamma$-convergence techniques to deal with the
regularization of nonlinear inverse problems may be found in \cite{Ron08}.

Since \cite{Riv-Bar-Ob}, it is well-known that handling the discretization for an inverse problem is not an easy and straightforward task. This difficulty is here
enhanced by the fact that we let change not only the mesh size parameter but also the regularization parameter and the number of measurements at the same time. In other words, we have to keep under control simultaneously the approximations due to the discretization, to the regularization and to the discrete measurements. This type of convergence analysis
has been carried out successfully for the electrical impedance tomography in a series of papers, \cite{R16} and \cite{FeR}. This is the first occasion when it is applied to scattering problems.

About the proof we point out the two main technical difficulties. The first one is to show that the direct problem is uniformly continuous, precisely H\"older continuous, with respect to the coefficients
. Here we improve, by making them quantitative, the stability results that has been obtained in \cite{Men-Ron}, \cite{LRX19} and  \cite{FR23} both for the acoustic and the electromagnetic scattering. The second challenge is the use of the $W^{1,\infty}$-seminorm as a regularization on $\sigma$. Given our mild regularity of $\Omega$, we have that a $W^{1,\infty}$-function is Lipschitz. However,
the Lipschitz constant is not that easy to compute, even if it can be bounded by a constant depending on $\Omega$ times the $W^{1,\infty}$-seminorm. Moreover,
 when we discretize a Lipschitz function using finite elements, it is not easy to control the Lipschitz constant of the interpolated function, see Remark~\ref{interLip}. However, this is possible if we discretize more regular functions, say $C^2$, see the second part
of Theorem~\ref{Ciarletestimteo}, in particular \eqref{Lipestimateinter}. In order to prove \eqref{Lipestimateinter}, an interesting technical result, Lemma~\ref{simplesso}, is proved.
To deal with $C^2$ functions and apply \eqref{Lipestimateinter}, we first need to extend the Lipschitz function to $\R^N$, using classical extension results, and suitably mollify it. Then one need to carefully choose the mollification parameter and the discretization parameter to obtain the correct approximation result of Lemma~\ref{Gammalimsuplemma}. Let us finally note that when extending a Lipschitz function outside an open set $\Omega$, if we wish to keep exactly the same $W^{1,\infty}$-seminorm we need to consider $\Omega$ convex. For this reason, in our main Theorem~\ref{mainthm}, we assume that any connected component of the support $\Omega$ of the penetrable obstacle is convex. However, this restrictive assumption may be easily dropped if we are satisfied with a just slightly weaker result, see Remark~\ref{finalremark}.

The plan of the paper is the following. In Section~\ref{prelsec} we collect the notation and well-known results that will be repeatedly used in the sequel. Section~\ref{contsec} contains the first important result of the paper, that is, H\"older continuity of the direct problem with respect to the coefficients describing the inhomogeneity, Theorem~\ref{holdercontthm} and Proposition~\ref{admprop}.
Finally, in Section~\ref{mainsec} we describe and prove our main result, Theorem~\ref{mainthm}, namely the approximation of the solution to the inverse problem by solutions to discrete and regularized variational problems.

\section{Notation and preliminaries}\label{prelsec}
Throughout the paper the integer $N\geq 2$ denotes the space dimension. Occasionally, $M$, $N_1$, $M_1$ denote positive integers.
For simplicity we drop the dependence of any constant on $N$.
With $\mathcal{H}^{s}$, $s\geq 0$, we denote the $s$-dimensional Hausdorff measure. On $\R^N$, $\mathcal{H}^{N}=\mathcal{L}$ is the $N$-dimensional Hausdorff measure or Lebesgue measure.

For any $x\in \R^N$ and any $r>0$, we denote with $B_r(x)$ the open ball with center $x$ and radius $r$. We often write simply $B_r$ to denote $B_r(0)$. For any vector $V\in\R^N$, $|V|$ denotes its Euclidean norm.

We denote with $\mathcal{M}^{N_1\times M_1} (\R)$ the space of real-valued $N_1\times M_1$ matrices, whereas
$\mathcal{M}^{N\times N}_{sym} (\R)$ is the space of symmetric matrices in $\mathcal{M}^{N\times N} (\R)$. With $\mathbb{I}_{N}$ we denote the $N\times N$ identity matrix. For any $A\in \mathcal{M}^{N_1\times M_1} (\R)$, we denote with $\|A\|$ the norm of $A$ as a linear operator from $\R^{M_1}$ to $\R^{N_1}$. Instead, with $|A|$ we denote the norm of $A$ as a vector in
$\R^{N_1M_1}$, that is, we identify $\mathcal{M}^{N_1\times M_1} (\R)$ with the Euclidean space $\R^{N_1M_1}$ in this case.

For any constants $\lambda_0\leq \lambda_1$, we call
$$\mathcal{M}(\lambda_0,\lambda_1)=\{A\in \mathcal{M}^{N\times N}_{sym} (\R): \lambda_0|\xi|^2\leq A\xi\cdot\xi\leq \lambda_1|\xi|^2\text{ for any }\xi\in\R^N\}.$$
Note that $\mathcal{M}(\lambda_0,\lambda_1)$ is convex and closed, with respect to any norm on $\mathcal{M}^{N\times N}_{sym} (\R)$.

If $F$ is an $\R^N$-valued measurable function on a measurable subset $E$ of $\R^M$, for any $1\leq p\leq +\infty$ we call
$$\|F\|_{L^p(E)}=\|F\|_{L^p(E,\R^N)}:=\|\,|F|\,\|_{L^p(E)}.$$
Analogously, if $A$ is an $\mathcal{M}^{N_1\times M_1}(\R)$-valued measurable function on a measurable subset $E$ of $\R^M$, for any $1\leq p\leq +\infty$ we call
$$\|A\|_{L^p(E)}=\|A\|_{L^p(E,\mathcal{M}^{N_1\times M_1}(\R))}:=\|\,\|A\|\,\|_{L^p(E)}$$
and
$$\trinorm A\trinorm_{L^p(E)}=\trinorm A\trinorm_{L^p(E,\mathcal{M}^{N_1\times M_1}(\R))}:=\|\,|A|\,\|_{L^p(E)}.$$
These two norms are clearly equivalent, however in certain cases one is best suited than the other.

We say that $\Omega\subset \R^N$ is a \emph{domain} if it is a connected open set.

\subsection{Lipschitz functions and Lipschitz open sets}
We begin with the following definition.
\begin{defi} 
Let $E$ be any subset of $\R^M$ and let $f:E\to\R^{M_1}$. We say that $f$ is \emph{Lipschitz continuous} on $E$ if there exists a constant $L\in \R$ such that
$$|f(x_1)-f(x_2)|\leq L|x_1-x_2|\quad\text{for any }x_1,\, x_2\in E.$$
We define the \emph{Lipschitz constant of} $f$ on $E$ as
$$\mathrm{Lip}(f)=\mathrm{Lip}(f,E):=\sup\left\{\frac{|f(x_1)-f(x_2)|}{|x_1-x_2|}:\ x_1,\, x_2\in E,\ x_1\neq x_2\right\}.$$

If $A$ is an $\mathcal{M}^{N_1\times M_1} (\R)$-valued function on a subset $E$ of $\R^M$,
$A$ is Lipschitz if and only if all its entries $a_{i,j}$, $i=1,\ldots,N_1$, $j=1,\ldots,M_1$, are and 
we call 
$$\mathrm{Lip}(A)=\mathrm{Lip}(A,E):=\sup\left\{\frac{|A(x_1)-A(x_2)|}{|x_1-x_2|}:\ x_1,\, x_2\in E,\ x_1\neq x_2\right\}.$$
\end{defi}

We make use of the following classical extension results.
\begin{theorem}\label{extensionthm}
Let $E$ be any subset of $\R^M$ and let $f:E\to\R^{M_1}$ be a Lipschitz function on $E$. Then there exists a Lipschitz function $\tilde{f}:\R^M\to\R^{M_1}$ on $\R^M$
such that $\tilde{f}(x)=f(x)$ for any $x\in E$ and $\mathrm{Lip}(\tilde{f},\R^M)=\mathrm{Lip}(f,E)$.
Moreover, if $f:E\to [\delta_0,\delta_1]$, for some constants $\delta_0\leq \delta_1$, we can choose $\tilde{f}$ such that $\delta_0\leq \tilde{f}\leq \delta_1$ on $\R^M$.

Clearly the same extension property holds for any $\mathcal{M}^{N_1\times M_1} (\R)$-valued function.
Moreover, if $A:E\to \mathcal{M}(\lambda_0,\lambda_1)$, we can choose the extension $\tilde{A}$ such that
$\tilde{A}:\R^M\to \mathcal{M}(\lambda_0,\lambda_1)$ as well.
\end{theorem}

\begin{proof}
See \cite{MS34} for the scalar case and \cite{Kir} for the vector-valued case. The fact that the extensions satisfy the same bounds as the original function is an easy consequence of the
convexity and closedness of the bound and the Hilbert structure of the codomain.
\end{proof}

\begin{defi}
We say that an open set $\Omega \subset \R^N$ is \emph{Lipschitz} or has a \emph{Lipschitz boundary} if for any $x\in \partial \Omega$ there exist a neighbourhood $U_x$ of $x$ and a Lipschitz function $\psi :\R^{N-1} \to \R$ such that, up to a rigid change of coordinates, we have
\begin{equation*}
    \Omega \cap U_x =\{ y=(y_1, \ldots, y_{N-1}, y_N)\,:\ y_N <\psi (y_1, \ldots, y_{N-1}) \}.
\end{equation*}

We say that a bounded open set $\Omega \subset \R^N$ belongs to the class $ \mathcal{A} (r,L,R)$ if $\Omega \subset B_R$ and its boundary is Lipschitz with constants $r$ and $L$ in the following sense: for any $x\in \partial \Omega$   we can choose $U_x = B_r(x)$ and $\psi$ with Lipschitz constant bounded by $L$.
\end{defi}

When $\Omega$ is a Lipschitz open set,  we usually denote with $\nu$ the exterior unit normal to $\partial \Omega$.
We also note that
for any bounded open set  with Lipschitz boundary, there exist constants $r, L$ and $R$ such that $\Omega \in \mathcal{A} (r,L,R)$.

\begin{defi} 
Let $\Omega\subset \R^M$ be an open set.

We say that $u\in W^{1,\infty}(\Omega,\R^{M_1})$ if $u\in L^{\infty}(\Omega,\R^{M_1})$ and
its distributional Jacobian $\nabla u$ belongs to $L^{\infty}(\Omega,\mathcal{M}^{M_1\times M} (\R))$. We define
$$|u|_{W^{1,\infty}(\Omega)}=\|\nabla u\|_{L^{\infty}(\Omega)}\quad\text{and}\quad\|u\|_{W^{1,\infty}(\Omega)}=\|u\|_{L^{\infty}(\Omega)}+|u|_{W^{1,\infty}(\Omega)}.$$
The same definition holds when $A$ is an $\mathcal{M}^{N_1\times M_1}(\R)$-valued function
on $\Omega$ by identifying $\mathcal{M}^{N_1\times M_1} (\R)$ with $\R^{N_1M_1}$, namely
$$|A|_{W^{1,\infty}(\Omega)}=\|\nabla A\|_{L^{\infty}(\Omega)}\quad\text{and}\quad\|A\|_{W^{1,\infty}(\Omega)}=\trinorm A\trinorm_{L^{\infty}(\Omega)}+|A|_{W^{1,\infty}(\Omega)},$$
with $\nabla A$ considered as a $ N_1M_1\times M$ matrix.

We say that $u\in BV(\Omega)$, that is, it is a function of bounded variation, if $u\in L^1(\Omega)$ and its distributional gradient $D u$ is a vector-valued Radon measure with finite total variation. We define
$$|u|_{BV(\Omega)}=TV(u)=|Du|(\Omega)\quad\text{and}\quad\|u\|_{BV(\Omega)}=\|u\|_{L^1(\Omega)}+|u|_{BV(\Omega)}.$$

If $F$ is an $\R^N$-valued function on $\Omega$, $F\in BV(\Omega,\R^N)$ if and only if all its components $F_i$, $i=1,\ldots,N$, belong to $BV(\Omega)$ respectively. Setting the vector
$\tilde{V}=(|F_1|_{BV(\Omega)},\ldots,|F_N|_{BV(\Omega)})$,
we call 
$|F|_{BV(\Omega)}=|\tilde{V}|$.

Similarly, if $A$ is an $\mathcal{M}^{N_1\times M_1}(\R)$-valued function 
on $\Omega$, $A\in BV(\Omega;\mathcal{M}^{N_1\times M_1} (\R))$ if and only if all its entries $a_{i,j}$, $i,\,j=1,\ldots,N$, belong to $BV(\Omega)$. Setting the matrix 
$\tilde{M}=\left[|a_{ij}|_{BV(\Omega)} \right]_{i,\,j=1}^{N_1,M_1}$,
we call 
$|A|_{BV(\Omega}=\|\tilde{M}\|$ or, equivalently, $|A|_{BV(\Omega}=|\tilde{M}|$.
\end{defi}

We need the following classical result whose proof is omitted.
\begin{prop}
Let $\Omega$ be an open set in $\R^N$ and let $f:\Omega\to \R^M$ be a Lipschitz function. Then $f\in W^{1,\infty}_{loc}(\Omega)$ and
$|u|_{W^{1,\infty}(\Omega)}\leq \mathrm{Lip}(f)$.

Let $\Omega\in\mathcal{A} (r,L,R)$ and let $u\in W^{1,+\infty}(\Omega)$. Then $u$ (up to a representative) is Lipschitz and there exists a constant $C$, depending on $r$, $L$ and $R$ only, such that
$\mathrm{Lip}(u)\leq C|u|_{W^{1,\infty}(\Omega)}$.

Finally, if $\Omega$ is a bounded convex open set, then $u\in W^{1,+\infty}(\Omega)$ if and only if $u$ (up to a representative) is Lipschitz and in this case
$\mathrm{Lip}(u)  =|u|_{W^{1,\infty}(\Omega)}$.
\end{prop}

\begin{rem}\label{noextensionrem}
Let $\Omega\in\mathcal{A} (r,L,R)$ and let $u\in W^{1,\infty}(\Omega)$. Then $u$ can be extended outside $\Omega$ to a $W^{1,\infty}(\R^N)$ function, which we still call $u$, such that
$\|u\|_{L^{\infty}(\R^N)}=\|u\|_{L^{\infty}(\Omega)}$ and 
$\mathrm{Lip}(u,\R^N)=|u|_{W^{1,\infty}(\R^N)}\leq C|u|_{W^{1,\infty}(\Omega)}$ for some constant $C$. However, unless $\Omega$ is convex, the constant $C$ might be strictly greater than $1$.
\end{rem}

We also need the following lower semicontinuity properties.
\begin{lem}\label{propLip2}
Let $E$ be any bounded set in $\R^M$, $E$ measurable and with positive measure.

Let $\{f_n\}_{n\in\N}$ be a sequence of Lipschitz functions on $E$ with values in $\R^{M_1}$ such that, for some constant $C$, we have
$\mathrm{Lip}(f_n,E)\leq C$ for any $n\in\N$. If $f_n$ converges to $f$ strongly in $L^1(E)$ as $n\to +\infty$, then $f$ (up to a representative) is Lipschitz on $E$ and the following lower semicontinuity holds
    \begin{equation*}
        \mathrm{Lip}(f,E) \leq \liminf_{n\to +\infty} \mathrm{Lip} (f_n,E).
    \end{equation*}
\end{lem}

\begin{proof} For any $n\in \N$, by uniform continuity we can extend uniquely  $f_n$ to $\overline{E}$, keeping the same Lipschitz constant. Convergence in $L^1(E)$, actually uniform boundedness in $L^1(E)$ would be enough, guarantees that there exists a constant $C_1$ such that the uniform norm of $f_n$ on $\overline{E}$ is bounded by $C_1$ for any $n\in\N$. By Ascoli-Arzel\`a Theorem, we have that up to a subsequence
$f_n$ converges uniformly to a function $\tilde{f}$ on $\overline{E}$. Clearly $\tilde{f}$ is Lipschitz on $\overline{E}$ and
is the (unique) extension by uniform continuity of $f$. By Urysohn, the  whole sequence $f_n$ converges to $\tilde{f}$ uniformly on $\overline{E}$ as $n\to +\infty$ and we can infer that
    \begin{equation*}
   \mathrm{Lip}(f,E)=      \mathrm{Lip}(\tilde{f},\overline{E}) \leq \liminf_{n\to +\infty} \mathrm{Lip} (f_n,\overline{E})=\liminf_{n\to +\infty} \mathrm{Lip} (f_n,E).
    \end{equation*}
    The proof is concluded.
   \end{proof}

\begin{lem}\label{propLip} Let $\Omega$ be any open set in $\R^N$.

Let $\{u_n\}_{n\in\N}$ be a sequence of $W^{1,\infty}_{loc}(\Omega,\R^{M_1})$ functions such that, for some constant $C$, we have
$|u_n|_{W^{1,\infty}(\Omega)}\leq C$ for any $n\in\N$. If $u_n$ converges to $u$ strongly in $L^1(\Omega)$ as $n\to +\infty$, then $u\in W^{1,\infty}_{loc}(\Omega,\R^{M_1})$ and the following lower semicontinuity holds
    \begin{equation*}
  |u|_{W^{1,\infty}(\Omega)} \leq \liminf_{n\to +\infty} |u_n|_{W^{1,\infty}(\Omega)}.
    \end{equation*}
\end{lem}

\begin{proof}
    Let $\Omega =\bigcup _{\ell =1}^\infty Q_\ell $ where each $Q_\ell$, $\ell\in\N$, is a cube.
    Then, by Lemma~\ref{propLip2}, for any $\ell\in\N$,
    \begin{multline*} |u|_{W^{1,\infty}(Q_{\ell})}=
    \mathrm{Lip}(u,Q_{\ell}) \\\leq \liminf_{n\to +\infty}\mathrm{Lip}(u_n,Q_{\ell}) = \liminf_{n\to +\infty}|u_n|_{W^{1,\infty}(Q_{\ell})}
     \leq \liminf_{n\to +\infty}|u_n|_{W^{1,\infty}(\Omega)}
\end{multline*}
and the proof is concluded.
\end{proof}

\subsection{Discretizations of a domain and finite element spaces}\label{finiteelsec}

A bounded open set $\Omega\subset\R^N$ is \emph{polyhedral} if
$\Omega$ coincides with the interior of $\overline{\Omega}$ and
its boundary is the finite union of cells, any \emph{cell} being the closure of an open connected subset of an $(N-1)$-dimensional hyperplane.

We consider standard conforming piecewise linear finite elements, see for instance \cite[Chapter~2]{C78}.
\begin{defi} Let $\Omega\subset\R^N$ be a polyhedral open set. 
A finite set $\mathcal{T}$ of subsets of $\overline{\Omega}$ is a \emph{triangulation} of $\overline{\Omega}$
if
\begin{itemize}
\item $\overline{\Omega}=\displaystyle{\bigcup_{K\in \mathcal{T}}K}$;
\item each $K\in \mathcal{T}$ is a closed
$N$-simplex with nonempty interior;
\item the intersection of two different elements of $\mathcal{T}$ is either empty or consists of a common face.
 \end{itemize}
\end{defi}

For any triangulation $\mathcal{T}$ of $\overline{\Omega}$, let the finite element space $Y^{\mathcal{T}}$ be given by
$$Y^{\mathcal{T}}=\{v\in C(\overline{\Omega}):\ v|_K\in P_1(K)\text{ for any }K\in\mathcal{T}\}$$
where $P_1(K)$ is the space of polynomials of order at most $1$ restricted to $K$.
We have that $Y^{\mathcal{T}}\subset C(\overline{\Omega})\cap H^1(\Omega)$ and that $Y^{\mathcal{T}}_{0}=\{v\in Y^{\mathcal{T}}:\ v|_{\partial\Omega}=0\}\subset  H^1_0(\Omega)$, see \cite[Theorem~2.2.3]{C78}.
We say that an $\R^N$-valued function $F$ belongs to $Y^{\mathcal{T}}_{\R^N}$ if all its components $F_i$, $i=1,\ldots,N$, belong to
$Y^{\mathcal{T}}$. Analogously, an $\mathcal{M}^{N\times N}(\R)$-valued function $A$ belongs to $Y^{\mathcal{T}}_{\mathcal{M}^{N\times N}(\R)}$ if all its entries $a_{i,j}$, $i,\ j=1,\ldots,N$, belong to
$Y^{\mathcal{T}}$.

For any $K\in\mathcal{T}$ we call
\begin{equation}\label{hKrhoKdefin}
h_K=\mathrm{diam}(K)\quad\text{and}\quad\rho_K=\sup
\{\mathrm{diam}(B):\ B\text{ is a ball contained in }K\}.
\end{equation}
We say that  $\mathcal{T}$ is a
\emph{regular triangulation} of $\overline{\Omega}$ with positive constants $s$ and $h$ if
\begin{equation}\label{regularfinite}
h_K\leq h\text{ and }h_K\leq s\rho_K\quad\text{for any }K\in\mathcal{T}.
\end{equation}

We call $\Pi_{\mathcal{T}}$ the associated interpolation operator
defined on $C(\overline{\Omega})$. When applied to vector-valued, or matrix-valued respectively, functions, it is assumed to be applied component by component, or entry by entry respectively.
We note that if $u\in C^0(\overline{\Omega};[\delta_0,\delta_1])$, then also $\Pi_{\mathcal{T}}(u)\in C^0(\overline{\Omega},[\delta_0,\delta_1])$.
Analogously,
if $A\in C^0(\overline{\Omega};\mathcal{M}(\lambda_0,\lambda_1))$, then $\Pi_{\mathcal{T}}(A)\in C^0(\overline{\Omega},\mathcal{M}(\lambda_0,\lambda_1))$ as well.

\begin{rem}\label{interLip}
Unfortunately the interpolation operator $\Pi_{\mathcal{T}}$, when applied to Lipschitz functions on $\Omega$, does not preserve the Lipschitz constant, that is,
$\mathrm{Lip}(\Pi_{\mathcal{T}}(u))$ might be strictly higher than $\mathrm{Lip}(u)$. A simple example shows this property. Let $K$ be the triangle in $\R^2$ with vertices
$(0,0)$, $(1,0)$, $(1/2,1/2)$. Then $u(x,y)=x\cos(y)$ is Lipschitz on $K$ with $\mathrm{Lip}(u,K)=1$. We have that $\Pi_{K}(u)(x,y)=x+(\cos(1/2)-1)y$ whose Lipschitz constant is
strictly greater than $1$.
\end{rem}

\begin{prop}\label{discretizable}
Let $\Omega\subset\R^N$ be a polyhedral domain or a polyhedral open set belonging to $\mathcal{A}(r,L,R)$. Then there exists a positive constant $s$, depending on $\Omega$, such that for any $h\in (0,1]$ there exists
a regular triangulation $\mathcal{T}_h$ of $\overline{\Omega}$ with constants $s$ and $h$. We call $Y^h=Y^{\mathcal{T}_h}$
and $\Pi_h=\Pi_{\mathcal{T}_h}$, for any $h\in (0,1]$.
\end{prop}

\begin{proof} We note that if $\Omega$ is a polyhedral open set belonging to $\mathcal{A}(r,L,R)$, then it is the union of a finite number of polyhedral domains whose closures are pairwise disjoint. Hence it is enough to consider the case of a polyhedral domain, where the result
essentially follows from the arguments of \cite{Edel1,Edel2}, see the proof of Remark~2.6 in \cite{FeR} for more details.
\end{proof}

We note that it is not easy to have a precise estimate of $s$ with respect to $\Omega$, for example it would be quite interesting to establish that if $\Omega\in \mathcal{A}(r,L,R)$, then $s$ depends on $r$, $L$ and $R$ only.

We conclude this preliminary part on discretization and finite elements with the following estimates on interpolations. The first part of the next theorem follows immediately from  \cite[Theorem~3.1.6]{C78}, for the second we provide a complete proof. 

\begin{theorem}\label{Ciarletestimteo}
Let $\mathcal{T}$ be a
regular triangulation of $\overline{\Omega}$ with positive constants $s$ and $h$.

Let us consider a real number $q$ such that $q>N/2$.
Then there exists a constant $\tilde{C}(q)$, depending on $q$ only, such that
for any $u\in W^{2,q}(\Omega,\R^{M_1})$ we have
\begin{multline}\label{Ciarletest}
\|u-\Pi_{\mathcal{T}}(u)\|_{L^{q}(\Omega)}\leq \tilde{C}(q)h^2\|D^2u\|_{L^q(\Omega)}\quad\text{and}\\
\|\nabla(u-\Pi_{\mathcal{T}}(u))\|_{L^{q}(\Omega,\R^N)}\leq \tilde{C}(q)sh\|D^2u\|_{L^q(\Omega)}.
\end{multline}

There exists a constant $C$, depending on $s$ only, such that for any $u\in C^2(\overline{\Omega},\R^{M_1})$ we have
\begin{equation}\label{Lipestimateinter}
|\Pi_{\mathcal{T}}(u)|_{W^{1,\infty}(\Omega)}\leq |u|_{W^{1,\infty}(\Omega)}+Ch\|D^2u\|_{L^{\infty}(\Omega)}.
\end{equation}

The same estimates, \eqref{Ciarletest} and \eqref{Lipestimateinter}, clearly hold for matrix-valued functions.
\end{theorem}

In order to prove \eqref{Lipestimateinter}, we need a  preliminary result. 

\begin{lem}\label{simplesso}
Given $s>0$, there exists a constant $C=C(s)$, depending on $s$ only, such that for any
closed
$N$-simplex with nonempty interior $K$ satisfying $h_K\leq s\rho_K$, $h_K$ and $\rho_K$ as in \eqref{hKrhoKdefin}, the following holds.

If we call $x_i$, for $i=0,\ldots,N$, the vertices of $K$ and $A_K$ the $N\times N$ matrix whose columns are $(x_1-x_0),\ldots,(x_N -x_0)$, then $A_K$ is invertible and 
\begin{equation}
\|A^{-1}(K)\| \leq C (s)h_K^{-1}.
\end{equation}
\end{lem}

\begin{proof}
By a simple rescaling argument, it is enough to show that $\|A^{-1}(K)\| \leq C (s)$ when $h_K=1$.
Let
$$\kappa=\{K\text{ is a simplex with }h_K=1,\text{ satisfying }h_K\leq s\rho_K \text{ and with }x_0=0\}.$$

First of all, we note that $\kappa$ is compact with respect to the following kind of convergence. For any sequence of simplexes $\{K_n\}_{n\in \N}$, with vertices $0$ and $x_j^n$, $j=1,\ldots,N$, for any $n\in \N$, and simplex $K$ with vertices $0$ and $x_j$, $j=1,\ldots,N$, we say that $K_n$ converges to $K$ as $n\to +\infty$ if $x_j^n\to x_j$ as $n\to +\infty$ for any $j=1,\ldots,N$.
It follows that $A_{K_n}\to A_K$ and it is not difficult to see that $K_n$ converges to $K$ in the Hausdorff distance, as $n\to +\infty$.
By convergence in the Hausdorff distance, we immediately infer that if $\{K_n\}_{n\in \N}\subset \kappa$ and $K_n\to K$, then $K\in \kappa$, thus $K$ has nonempty interior and $A_K$ is invertible, that is,
$\mathrm{det}(A_K)>0$ and $\|A^{-1}_K\|<+\infty$.
We conclude that $\kappa$ is indeed compact with respect to this kind of convergence. We immediately infer that there exist positive constants $c_1$ and $c_2$ such that
$c_1\leq \mathrm{det}(A_K)\leq c_2$ for any $K\in \kappa$, therefore if a sequence $\{K_n\}_{n\in \N}\subset \kappa$ converges to $K$ as $n\to +\infty$, we also have
$A_{K_n}^{-1}\to A_K^{-1}$. Again by compactness, it follows that
$$C(s):=\sup\{\|A^{-1}(K)\|:\ K\in \kappa\}<+\infty$$
and the proof is concluded.
\end{proof}

We now conclude the proof of the second part of Theorem~\ref{Ciarletestimteo}.

\begin{proof}[Proof of Theorem~\ref{Ciarletestimteo}]
It is clearly enough to restrict ourselves to a single $K\in \mathcal{T}$. Without loss of generality, if $h_K$ is the diameter of $K$, we can assume that $K=h_K K_0$ with $K_0\in\kappa$, $\kappa$ as in the proof of Lemma~\ref{simplesso}. We also assume, again without loss of generality, that $u(0)=0$. We let $0$ and $x_1,\ldots,x_N$ be the vertices of $K_0$. 
Then, on $K$, $\Pi_{\mathcal{T}}(u)(x)=\tilde{B}x$ for some matrix $\tilde{B}\in\mathcal{M}^{M_1\times N}$ and any $x\in K$, and $|\Pi_{\mathcal{T}}(u)|_{W^{1,\infty}(K)}=\|\tilde{B}\|$.
We have that $\tilde{B}$ solves the following equation
    \begin{equation*}
h_K\tilde{B}A_{K_0}  = 
\left[ u (h_Kx_1) 
\cdots 
u (h_Kx_N) \right]
\end{equation*}
where $u (h_Kx_j)$, $j=1,\ldots,N$,  is an $M_1$-column vector.
For any $j=1,\ldots,N$ and an absolute constant $C_0$,
$$|u(h_Kx_j)- h_K \nabla u(0)[x_j]| \leq C_0 h_K^2\|D^2u\|_{L^{\infty}(K)}.$$
Consequently, for another absolute constant $\tilde{C}_0$,
$$\|h_K(\tilde{B}-\nabla u(0))A_{K_0}\|\leq \tilde{C}_0 h_K^2\|D^2u\|_{L^{\infty}(K)}$$
thus
  \begin{equation*}
        \|\tilde{B}-\nabla u(0)\|\leq \|A^{-1}(K_0)\|\tilde{C}_0h_K \|D^2u\|_{L^{\infty}(K)}\leq
        \tilde{C}_0C(s)h_K  \|D^2u\|_{L^{\infty}(K)},
    \end{equation*}
    where $C(s)$ is the constant of Lemma~\ref{simplesso}. The proof is concluded.
\end{proof}

\subsection{$\Gamma$-convergence}
We recall the definition of $\Gamma$-convergence, see \cite{DM93} and \cite{B02} for a more thorough introduction. Let $(X, d)$ be a metric space. A sequence $F_n : X \to [-\infty , +\infty ]$, $n \in \N$, $\Gamma$-converges as $n\to +\infty$ to a function $F_n : X \to [-\infty , +\infty ]$ if for every $x\in X$ we have 
\begin{enumerate}
    \item\label{Gammaliminf}
    for every sequence $\{x_n\}_{n\in \N}$ converging to $x$ we have
    \begin{equation*}
        F(x) \leq \liminf_n F_n(x_n);
    \end{equation*}
    \item\label{Gammarecovery}
     there exists a sequence $\{x_n\}_{n\in \N}$ converging to $x$ such that
    \begin{equation*}
        F(x) =\lim _n F_n(x_n).
    \end{equation*}
\end{enumerate}
The function F is the $\Gamma$-limit of the sequence $\{F_n\}_{n\in \N}$ as $n\to +\infty$ and is denoted by $F = \Gamma $-$\displaystyle\lim_n F_n$.
Property (\ref{Gammaliminf}) above is called $\Gamma$-$\liminf$ inequality and property (\ref{Gammarecovery}) is called the existence of a recovery sequence.

We say that the sequence $\{F_n\}_{n\in \N}$ is \emph{equicoercive} if there exists a compact set $K\subset X$ such that $\displaystyle \inf_K F_n = \inf_X F_n$ for any $n\in \N$.

The following theorem, known as the Fundamental Theorem of $\Gamma$-convergence, illustrates the motivations for the definition of this variational convergence. For its proof we refer, for instance, to \cite[Theorem 1.21]{B02}.
\begin{theorem}\label{fundGammathm}
     Let $(X, d)$ be a metric space and let $F_n : X \to [-\infty , +\infty ]$, $n \in \N$, be a sequence of functions defined on $X$. If 
     the sequence $\{F_n\}_{n\in \N}$ is equicoercive
     and $\displaystyle F = \Gamma $-$\lim_n F_n$, then $F$ admits a minimum over $X$ and we have
\begin{equation*}
    \min_X F = \lim _n \inf _X F_n.
\end{equation*}
Furthermore, if $\{x_n\}_{n\in \N}$ is a sequence converging to $x\in X$ such that $\displaystyle \lim_n F_n(x_n) = \lim_n \inf_X F_n$, then $x$ is a minimum point for $F$.
\end{theorem}

A family of functions $F_\varepsilon$, defined for every $\varepsilon>0$, $\Gamma$-converges to a function $F$ as $\varepsilon \to 0^+$ if for
every sequence $\{\varepsilon _n\}_{n\in \N}$ of positive numbers converging to $0$ as $n\to +\infty$, we have $F = \Gamma $-$\displaystyle \lim_n F_{\varepsilon_n}$.

\subsection{The Helmholtz equation and radiating solutions}
We consider the \emph{Helm\-holtz equation}, or \emph{reduced wave equation}, $\Delta u+k^2 u=0$ in an open subset of $\R^N$. Here $k>0$ is the \emph{wavenumber}. We review a few well-known facts about the Helmholtz equation, for details see for instance \cite{CK98}. The fundamental solution to the Helmholtz equation is given by
\begin{equation*}
    \Phi_k (x,y)=\frac i 4 \left( \frac k {2\pi |x-y|}\right)^{(N-2)/2} H^{(1)}_{(N-2)/2} (k|x-y|),\quad x,y \in \R^N,\ x\ne y.
\end{equation*}
Here $H^{(1)}_{\alpha}$, for any $\alpha\geq 0$, is the Hankel function of order $\alpha$.

Let $R_0>0$. A solution $u$ to the Helmholtz equation in $\R^N\setminus \overline{B_{R_0}}$ is \emph{radiating} or \emph{outgoing} if it satisfies the so-called \emph{Sommerfeld radiation condition}
$$\lim_{r\to +\infty } r^{(N-1)/2} \left(\frac{\partial u}{\partial r} - kiu\right) =0, \quad r=|x|$$
where the limit has to hold uniformly with respect to all directions $\hat{x}=x/|x|\in \SSS^{N-1}$.
A radiating solution $u$ has the following asymptotic behavior 
\begin{equation*}
    u(x)=\frac{e^{ik|x|}}{|x|^{(N-1)/2}} \left( u_\infty (\hat x )+O\left(\frac{1}{|x|}\right)\right), \quad |x|\to +\infty ,
\end{equation*} where $\hat x=x/|x|$ and the limit has to hold uniformly with respect to $\hat x\in\SSS^{N-1}$.
 The function $u_\infty:\SSS^{N-1}\to \C$ is called the \emph{far-field pattern} of the radiating solution $u$.
The far-field is given by
\begin{equation}\label{firstffrepr}
    u_\infty (\hat x)= \lim_{r\to +\infty } \frac{r^{(N-1)/2}}{e^{ikr}} u(r\hat x) \quad \text{for any }\hat{x}\in \SSS^{N-1}.
\end{equation}

We recall that a radiating solution $u$ satisfies the following Green representation for any $R>R_0$
\begin{equation}\label{Greenrep}
u(x)=\int_{\partial B_{R}}\left( u(y)\frac{\partial \Phi_k(x,y)}{\partial\nu(y)}-\frac{\partial u }{\partial\nu}(y)\Phi_k(x,y)\right)d\mathcal{H}^{N-1}(y),\quad x\in \R^N\setminus \overline{B_{R}}
\end{equation}
where $\nu$ is the exterior unit normal to $\partial B_{R}$.

By \eqref{Greenrep}, \eqref{firstffrepr}, and 
the following asymptotic behavior of the Hankel functions for $\alpha\geq 0$, see for example \cite[formula (3.59)]{CK98} or \cite[Chapter~5]{Leb},
\begin{equation*}
    \begin{array}{ll}
        H^{(1)}_{\alpha} (t) = \sqrt{\frac{2}{\pi t}} e^{ i(t-\frac{\alpha\pi}{2}-\frac{\pi}{4})} \left(1+O\left(\frac{1}{t} \right) \right), & t \to +\infty\\
        (H^{(1)}_{\alpha})' (t)  = \sqrt{\frac{2}{\pi t}} e^{ i(t-\frac{\alpha\pi}{2}+\frac{\pi}{4})} \left(1+O\left(\frac{1}{t} \right) \right), & t \to +\infty,
    \end{array}
\end{equation*}
we deduce that for any $R>R_0$
\begin{multline}\label{farfield}
    u_\infty (\hat x) \\ =\frac{e^{-i \frac{(N-3)\pi}{4}}}{4\pi}\left(\frac{k}{2\pi}\right)^{\frac{N-3}2}
    \int_{\partial B_{R}}  \left(u(y)  \frac{\partial e^{-ik \hat x \cdot y}}{\partial\nu(y)}   -\frac{\partial u }{\partial \nu }(y)e^{-ik \hat x \cdot y}\right) d\mathcal{H}^{N-1}(y).
\end{multline}
From \eqref{farfield} it is easy to infer that $u_\infty$ is an analytical function on $\mathbb{S}^{N-1}$.

\subsection{Scattering problems}

We consider the scattering problem due to the presence of an inhomogeneity in an otherwise homogeneous isotropic medium in the acoustic time-harmonic framework. Throughout the paper we keep fixed the following parameters $R_0>0$, $0<\lambda_0 < \lambda_1$, $0<\delta_0<\delta_1$ and $ 0<k_0<k_1$. We begin with a general definition of the coefficients describing the inhomogeneity.
We call $X=L^1(B_{R_0},\mathcal{M}^{N\times N}(\R))\times L^1(B_{R_0)}$ and we endow it with the following norm
$$\|(\sigma,\eta)\|_X=\|\sigma\|_{L^1(B_{R_0})}+\|\eta\|_{L^1(B_{R_0})}.$$

\begin{defi}\label{class}
We say that $(\sigma , \eta)\in X$ belongs to $\tilde{X}=(R_0,\lambda_0, \lambda_1 ,\delta_0 , \delta_1)$ if
the following holds. First, $\sigma \in L^\infty (\R^N, M^{N\times N}_{sym} (\R) )$ is a \emph{uniform elliptic tensor} in $\R^N$ with constants $\lambda_0$ and $\lambda_1$, that is, 
\begin{equation}\label{5.1FR}
    \lambda_0 |\xi |^2 \leq \sigma (x) \xi \cdot \xi \leq \lambda_1 |\xi|^2 \quad \text{for a.e. } x\in \R^N \text{ and for any }\xi \in \R^N.
\end{equation}
Second,
$\eta \in L^\infty (\R^N)$ satisfies
\begin{equation}\label{boundseta}
     \delta_0 \leq \eta (x) \leq \delta_1 \quad \text{for a.e. } x\in \R^N.
\end{equation}
We assume that the space is homogenous isotropic outside the ball $\overline{B_{R_0}}$, that is,
\begin{equation}\label{5.9FR}
    \sigma \equiv \mathbb{I}_N \quad\text{and} \quad \eta \equiv 1 \quad \text{in } \R^N \setminus \overline{B_{R_0}}.
\end{equation}
Finally, we say that
\begin{equation*}
    (\sigma , \eta , k) \in \tilde{\nn} =\tilde{\nn} (R_0,\lambda_0, \lambda_1 ,\delta_0 , \delta_1 , k_1)
\end{equation*}
if $(\sigma,\eta)\in \tilde{X}$ and
the number $k\in \R$ satisfies 
\begin{equation}\label{5.11FR}
0< k\leq k_1.
\end{equation}
\end{defi}

The scattering problem is the following. 
Let $u^i$ be an entire solution to the Helmholtz equation $\Delta u+k^2 u=0$, with $k>0$. We restrict ourselves to the case when $u^i$ is a \emph{plane wave}, namely
\begin{equation}\label{incidentfield}
u^i(x)=u^i(k,d;x)=e^{k\rmi d\cdot x}\quad\text{for any }x\in \R^N,
\end{equation}
where $d\in \SSS^{N-1}$ is the \emph{direction of propagation}. The plane wave $u^i$ is called the \emph{incident field} and we aim to solve the following scattering problem
\begin{equation}\label{pbScattering}
\left\{
\begin{array}{ll}
\mathrm{div} (\sigma \nabla u) +k^2 \eta u= 0 &\text{in } \R^N \\
u=u^i + u^s &\text{in } \R^N \\
\displaystyle \lim_{r\to +\infty } r^{\frac{N-1}{2}} \left(\frac{\partial u^s}{\partial r} - kiu^s\right) =0& r=|x|.
\end{array}
\right.
\end{equation} 
Here the function $u=u(\sigma,\eta;k,d;\cdot)$ is called the \emph{total field} and $u^s=u^s(\sigma,\eta;k,d;\cdot)$ is called the \emph{scattered field}. About the coefficients, we assume that $(\sigma,\eta,k)\in \tilde{\mathcal{N}}$, with $\tilde{\mathcal{N}}$ as in Definition~\ref{class}.
Hence both $u$ and $u^s$ satisfy the Helmholtz equation $\Delta u+k^2 u=0$ outside $\overline{B_{R_0}}$
and $u^s$ is radiating.

We call
$$u_{\infty}(\sigma,\eta;k,d;\cdot)=u_{\infty}:\mathbb{S}^{N-1}\to\C $$
the far-field pattern of the scattered field $u^s=u^s(\sigma,\eta;k,d;\cdot)$.
We note that the following \emph{reciprocity relation} holds
\begin{equation}\label{recrelation}
u_{\infty}(\sigma,\eta;k,d;\hat x)=u_{\infty}(\sigma,\eta;k,-\hat x;-d)\quad\text{for any }d,\, \hat{x}\in\mathbb{S}^{N-1}.
\end{equation}

The following definition is needed to establish existence and uniqueness of a solution to \eqref{pbScattering}.

\begin{defi}
For any domain $D$ in $\R^N$, let $\sigma$ be a uniformly elliptic tensor in $D$ and $\eta \in L^\infty (D)$ satisfy \eqref{boundseta} in $D$.
The Helmholtz type equation $\mathrm{div}(\sigma \nabla u )+k^2\eta u=0$, with $k\geq 0$,
satisfies the \emph{unique continuation property} (UCP) in $D$ if for any solution $u$ that vanishes in an open not empty subset $D_1 \subset D$, it follows that $u\equiv 0$ in all of $D$.    
\end{defi}

About sufficient condition for the UCP to hold for elliptic equations we refer, for instance, to \cite{Aleetal}.

By the classical Phillips method one can deduce the following existence and uniqueness result.
\begin{prop}
Let us assume that $\sigma$ and $\eta$ are such that $\mathrm{div}(\sigma \nabla u )+k^2\eta u=0$
satisfies the unique continuation property in $\R^N$. Then the scattering problem \eqref{pbScattering} admits a unique solution.
\end{prop}

\subsection{Regularity results}
We collect a few classical regularity results for elliptic equations that we use in the sequel.
Throughout this subsection, $\sigma$ is a uniformly elliptic tensor in $\R^N$ satisfying \eqref{5.1FR}, $\eta\in L^{\infty}(\R^N)$ satisfying 
\eqref{boundseta} and $0\leq k\leq k_1$ for some $k_1>0$

We begin with the following Meyers theorem.

\begin{theorem}\cite[Theorem~2]{M63}, \label{theorem1.4}
    Let $r$, $L$ and $R$ be positive numbers and let $0<\lambda_0 <\lambda_1$.
    Then  there exists a constant $Q_1>2,$ depending on $r$, $L$, $R$, $\lambda_0$ and $\lambda_1$ only, such that for any $p \in [2,Q_1]$ the following holds.
    
    Let $\Omega \in \mathcal{A} (r,L,R)$. Then 
    there exists $C=C(p)$ depending on $r$, $L$, $R$, $\lambda_0$, $\lambda_1$ and $p$ only, such that for any $\phi \in W^{1,p} (\Omega )$, $f\in L^2(\Omega)$ and $F\in L^{p}(\Omega;\R^N),$ if $u$ solves
\begin{equation}\label{2.24}
\left\{
\begin{array}{ll}
\mathrm{div}(\sigma \nabla u)=f+ \mathrm{div}(F) &\text{in } \Omega \\
u=\phi &\text{on } \partial \Omega,
\end{array}
\right.
\end{equation}
then we have that
    \begin{equation}\label{23.new} 
        \|u\|_{W^{1,p} (\Omega )} \leq C (\|\phi \|_{W^{1,p}(\Omega )}+ \|f\|_{L^2 (\Omega )} +\|F\|_{L^p(\Omega)}).
    \end{equation}
\end{theorem}

Next we write the Caccioppoli inequality and a classical boundedness result.

\begin{lem}
\label{caccioppoli}
Let $\Omega$ be any open set and 
    let $v\in W^{1,2}(\Omega)$ be a solution to $\mathrm{div}(\sigma \nabla v) +k^2\eta v =f + \mathrm{div}(F)$ in $\Omega$, for some $f\in L^2(\Omega)$ and $F\in L^2(\Omega;\R^N)$. Then, for any $\Omega ' \Subset \Omega$ there exists a constant $C$, depending on $\mathrm{dist}(\Omega',\partial\Omega)$, $\lambda_0$, $\lambda_1$, $\delta_1$ and $k_1$ only, such that the following Caccioppoli inequalty
    \begin{equation}\label{disuguaglianza}
        \|\nabla v \|_{L^2(\Omega ')} \leq C( \| v \|_{L^2(\Omega )}  + \|f\|_{L^2(\Omega )} + \|F\|_{L^2(\Omega )}).
    \end{equation}
  holds  and, provided $f\equiv 0$ and $F\equiv 0$ in $\Omega$, we also have
     \begin{equation}\label{boundednessapriori}
        \|v\|_{L^\infty (\Omega ')} \leq C \|v\|_{L^2 (\Omega )}.
    \end{equation}   
    \end{lem}

\begin{proof}
    Let $\chi\in C^\infty _0 (\Omega)$ be such that $0\leq \chi\leq 1$ in $\Omega$ and $\chi \equiv 1$ in $\Omega '$.
    The Caccioppoli inequality \eqref{disuguaglianza} follows by standard arguments using $\chi ^2 v$ as a test function.

About 
\eqref{boundednessapriori}, see, for instance, Theorem~8.24 in \cite{GT01}.
\end{proof}

\begin{lem}\label{Lemma5marzo}
    Let $\Omega$ be any open set and let $u\in W^{1,2}(\Omega)$ be a solution to $\Delta u=f$ in $\Omega$ for some $f\in W^{\ell , 2}(\Omega)$ with $\ell \geq 1.$  Then for any 
   $\Omega ' \Subset \Omega$, we have that $u\in W^{\ell +2, 2}(\Omega ')$  and there exists a constant $C$, 
 depending on $\mathrm{dist}(\Omega',\partial\Omega)$ and $\ell$ only, such that  
    \begin{equation*}
        \|u\|_{W^{\ell +2, 2}(\Omega ')} \leq C(\|u\|_{L^2(\Omega )} + \|f\|_{W^{\ell, 2}(\Omega )}).
    \end{equation*}
    \end{lem}

\begin{proof}
    It follows immediately by Lemma~\ref{caccioppoli} and Theorem~8.10 in \cite{GT01}.
\end{proof}

We conclude this section with a regularity result concerning the far-field pattern.

\begin{rem}\label{regularityfarfield}
Let us fix $R_0>0$ and let $k$ be such that
\begin{equation}\label{condonk}
k_0\leq k \leq k_1\text{ if }N=2\quad\text{or}\quad 0<k\leq k_1\text{ if }N\geq 3
\end{equation}
for some constants $0<k_0<k_1$.
Let $w$ be a radiating solution to $\Delta w+k^2w=0$ in $\R^N\setminus \overline{B_{R_0}}$.
Let $w_{\infty}:\SSS^{N-1}\to\C$ be its far-field pattern. Let
$A$ be the annulus given by $A=B_{R_0+2}\setminus \overline{B_{R_0}}$.
Then there exists a constant $C$, depending on $R_0$, $k_0$ if $N=2$, and $k_1$ only, such that
\begin{equation}\label{boundfarfield}
|w_{\infty}(\hat{x})|\leq C\| w  \|_{L^2(A)}\quad\text{for any }\hat{x}\in \SSS^{N-1}
\end{equation}
and
\begin{equation}\label{Lipfarfield}
|w_{\infty}(\hat{x}_1)-w_{\infty}(\hat{x}_2)|\leq  C\| w  \|_{L^2(A)}\quad\text{for any }\hat{x}_1,\, \hat{x}_2\in \SSS^{N-1}.
\end{equation}

In fact, by standard elliptic estimates, namely by Lemma~\ref{caccioppoli} and repeated applications of Lemma~\ref{Lemma5marzo},
 we can find a constants $C_1$, depending on $k_1$ and $R_0$ only, such that
 \begin{equation}\label{boundaryest}
 \|w \|_{ L^{\infty}(\partial B_{R_0+1})} +\left\|\frac{\partial w}{\partial\nu} \right\|_{ L^{\infty}(\partial B_{R_0+1})} \leq C_1 \|w\|_{ L^{2}(A)}.
\end{equation}
Then \eqref{boundfarfield} and \eqref{Lipfarfield} follows from the representation \eqref{farfield} with $R=R_0+1$.
\end{rem}

\section{Continuity of the direct problem with respect to the inhomogeneity}\label{contsec}

We are interested in the continuity of the solutions to \eqref{pbScattering} with respect to the coefficients $\sigma$ and $\eta$. For our purposes we require a quantitative version of continuity, namely we prove a H\"older type continuity. We need to restrict our class of coefficients, starting with the following abstract formulation.

\begin{defi}\label{admissibleclass}
We say that $\nn_{ad}\subset\tilde{\nn}$ is an \emph{admissible class} if for any $s\geq R_0+1$ there exists $C=C(\nn_{ad},s)$, depending only on the class $\nn_{ad}$ and on $s$, such that the following holds. For any $(\sigma , \eta , k)\in \nn_{ad}$ and any $f\in L^2(\R^N)$ and $F\in L^2(\R^N;\R^N)$, with
$f\equiv 0$ and $F\equiv 0\in\R^N$ almost everywhere in $\R^N\setminus B_{R_0}$, the problem
    \begin{equation}\label{pb1}
\left\{
\begin{array}{ll}
\mathrm{div}(\sigma \nabla w) +k^2 \eta w= f+\mathrm{div}(F) &\text{in } \R^N \\
\displaystyle \lim_{r\to +\infty } r^{\frac{N-1}{2}} \left(\frac{\partial w}{\partial r} - kiw\right) =0  &r=|x|,
\end{array}
\right.
\end{equation}
admits a solution $w$ and for any such solution $w$ we have
\begin{equation}\label{aprioriboundadm}
    \|w\|_{L^2(B_s)} \leq C(\|f\|_{L^2(B_{R_0})} + \|F\|_{L^2(B_{R_0})}).
\end{equation}
\end{defi}

\begin{rem}
If $(\sigma , \eta , k)\in \nn_{ad}$, then \eqref{pb1} and \eqref{pbScattering} admit a unique solution. Moreover, if $u=u^i+u^s$ solves \eqref{pbScattering} with $u^i$ given by \eqref{incidentfield} for some $d\in \SSS^{N-1}$, we have that $u^s$ solves \eqref{pb1} with
$f=-k^2\eta u^i$ and $F=-\sigma\nabla u^i$. Hence, for any $s\geq R_0+1$
there exists $C_1=C_1(\nn_{ad},s)$, depending only on the class $\nn_{ad}$ and on $s$, such that 
\begin{equation}\label{aprioribound}
    \|u\|_{L^2(B_s)}+\|u^s\|_{L^2(B_s)} \leq C_1.
\end{equation}
\end{rem}

The first result we prove is the following.

\begin{theorem}\label{Teorema1}
Let $\nn_{ad}$ be an admissible class as in Definition~\ref{admissibleclass}.

For any $(\sigma , \eta , k ) \in \nn_{ad}$ and any $d\in \SSS ^{N-1}$, let $u$ be the solution to \eqref{pbScattering} with $u^i$ given by \eqref{incidentfield}.

For any $s\geq R_0+1$,
let $Q_1>2$ and $C(Q_1)$ be as in Theorem~\ref{theorem1.4} for $\sigma$ and $\Omega=B_s$, thus $Q_1=Q_1(s)$ and $C(Q_1)$ depend on $s$,
$\lambda_0$ and $\lambda_1$ only.

Then
there exists $C_2=C_2(\nn_{ad},s)$, depending only on $\nn_{ad}$ and $s$, such that  
\begin{equation}\label{5.12}
    \|u\|_{W^{1,Q_1}(B_s)} \leq C_2.
\end{equation} 
\end{theorem}

\begin{proof}
Let us fix $s\geq R_0+1$. We call $A$ the annulus given by $A=B_{s+1}\setminus \overline{B_{s-1}}$.
We have that $u$ solves
 \begin{equation*}
\left\{
\begin{array}{ll}
\mathrm{div}(\sigma  \nabla u) =-k^2 \eta  u =:f &\text{in } B_s \\
u=\phi &\text{on } \partial B_s 
\end{array}
\right.
\end{equation*}
where $\phi$ has the same trace of $u$ on $\partial B_s$.
  By \eqref{aprioribound} we have that  $\|f\|_{L^2(B_s)} \leq k_1^2\delta_1C_1(\nn_{ad},s)$ and $\|u\|_{L^2(A)}\leq C_1(\nn_{ad},s+1)$.
  
 We note that $u$ solves  $\Delta u +k^2u=0$  in $A$. By standard elliptic estimates, namely by Lemma~\ref{caccioppoli} and repeated applications of Lemma~\ref{Lemma5marzo},
 we can find constants $C_3$ and $\tilde{C}_3$, depending on $k_1$, $s$ and $Q_1(s)$ only, such that
 \begin{equation*}
 \|u \|_{ W^{1-1/Q_1, Q_1}(\partial B_s)} \leq C_3 \|u\|_{ L^{2}(A)}
\end{equation*}
and we can choose $\phi$ such that
$$
 \|\phi \|_{ W^{1, Q_1}( B_s)} \leq \tilde{C}_3 \|u\|_{ L^{2}(A)}.
$$
The conclusion immediately follows by using Theorem~\ref{theorem1.4}.
 \end{proof}

We are ready to prove our H\"older continuity result.

\begin{theorem}\label{holdercontthm}
Let $\nn_{ad}$ be an admissible class as in Definition~\ref{admissibleclass}.

For any $(\sigma_1 , \eta_1 , k )$ and $(\sigma_2,\eta_2,k) \in \nn_{ad}$ and any $d\in \SSS ^{N-1}$, 
let $u^i$ be given by \eqref{incidentfield} and let $u_j$, for $j=1,2$, be the solution to
\begin{equation}\label{pbScattering3}
\left\{
\begin{array}{ll}
\mathrm{div} (\sigma_j \nabla u_j) +k^2 \eta_j u_j= 0 &\text{in } \R^N \\
u_j=u^i + u^s_j &\text{in } \R^N \\
\displaystyle \lim_{r\to +\infty } r^{\frac{N-1}{2}} \left(\frac{\partial u^s_j}{\partial r} - kiu^s_j\right) =0& r=|x|.
\end{array}
\right.
\end{equation}

Then for any $s\geq R_0+1$ there exists $\tilde{C}=\tilde{C}(\nn_{ad},s)$, depending on $\nn_{ad}$ and $s$ only, such that
\begin{equation}\label{dis1}
    \|u_2-u_1\|_{L^2(B_s)} \leq\tilde{C}\left(\|\eta _2 - \eta _1\|_{L^1(B_{R_0})}^{1/2} + \|\sigma _2-\sigma _1\|_{L^1(B_{R_0})} ^\beta\right),
\end{equation}
where
\begin{equation*}
    \beta = \frac{Q_1-2}{2Q_1},
\end{equation*}
and $Q_1=Q_1(s)>2$ is the same as the one in Theorem~\ref{Teorema1}.

Moreover, provided \eqref{condonk} holds for the class $\nn_{ad}$, there exists a constant $\tilde{C}_1=\tilde{C}_1(\nn_{ad})$, depending on  $\nn_{ad}$ only, such that, for any $d\in \SSS^{N-1}$ and any $\hat{x}\in\SSS^{N-1}$, we have
\begin{multline}\label{holderfarfield}
|u_{\infty}(\sigma_2,\eta_2;k,d;\hat{x})-u_{\infty}(\sigma_1,\eta_1;k,d;\hat{x})|\\\leq \tilde{C}_1\left(\|\eta _2 - \eta _1\|_{L^1(B_{R_0})}^{1/2} + \|\sigma _2-\sigma _1\|_{L^1(B_{R_0})} ^\beta\right).
\end{multline}
\end{theorem}

\begin{proof}
    
If we consider $w=u_2-u_1$, by an easy computation we obtain that
$$ \mathrm{div}(\sigma _2 \nabla w) +k^2\eta _2 w=- \mathrm{div}((\sigma_2-\sigma_1)\nabla u_1) -k^2(\eta _2-\eta_1)u_1\quad \text{in }\R^N$$
and
$$
\displaystyle \lim_{r\to +\infty } r^{\frac{N-1}{2}} \left(\frac{\partial w}{\partial r} - kiw\right) =0,\quad r=|x|.
$$
Calling $f=-k^2(\eta _2 - \eta _1) u_1$ and $F=-(\sigma _2-\sigma _1) \nabla u_1$, we have that $w$ satisfies \eqref{pb1} with $\sigma$ and $\eta$ replaced by $\sigma_2$ and $\eta_2$ respectively. By \eqref{aprioriboundadm},
\begin{multline*}
 \|u_2-u_1\|_{L^2(B_s)} \\\leq C\left(k_1^2\|u_1\|_{L^\infty(B_{R_0})}\|\eta _2 - \eta _1\|_{L^2(B_{R_0})} + \|\nabla u_1\|_{L^{Q_1}(B_{R_0})}\|\sigma _2-\sigma _1\|_{L^q (B_{R_0})} \right)
 \end{multline*}
where $q$ is such that
\begin{equation*}
    \frac{1}{q} + \frac{1}{Q_1} =\frac 1 2 .
\end{equation*}
We can conclude the proof of \eqref{dis1} by using \eqref{boundednessapriori}, via \eqref{aprioribound}, and Theorem~\ref{Teorema1} and by the following argument. We have
$$\|\sigma _2-\sigma _1\|_{L^q (B_{R_0})}\leq (2\lambda_1)^{1-\beta }\|\sigma _2-\sigma _1\|^\beta _{L^1 (B_{R_0})}$$
 where $\beta =1/q,$ hence
\begin{equation*}
    \beta = \frac{Q_1-2}{2Q_1},
\end{equation*}
and
$$\|\eta_2 -\eta _1\|_{L^2 (B_{R_0})}\leq (2\delta_1)^{1/2} \|\eta _2 - \eta _1\|_{L^1(B_{R_0})}^{1/2}.$$

About \eqref{holderfarfield}, it follows from Remark~\ref{regularityfarfield}, in particular from \eqref{boundfarfield}.
\end{proof}

For the sake of completeness, we provide a quite general admissible class that we borrow from \cite{FR23}.

\begin{prop}\label{admprop}
We have that the class $\nn$ defined in \cite[Definition~5.12]{FR23} is an admissible class in the sense of Definition~\ref{admissibleclass}.
Moreover, for such a class $\nn$ we have that the function $u_{\infty}(\sigma,\eta;k,d;\hat{x})$ is continuous with respect to the variable $k$ in $(0,k_1]$.

\end{prop}

\begin{proof} We just sketch the proof, since the result is essentially that of Theorem~5.14 in \cite{FR23}, which in turn is based on arguments previously developed in \cite{Men-Ron} and \cite{LRX19}.

Let $s\geq R_0+1$. 
  We argue by contradiction. Let us assume that, for any $n\in\N$ there exist $(\sigma_n,\eta_n,k_n)\in \nn_{ad}$, $f_n\in L^2(B_{R_0})$ and 
  $F_n\in L^2(B_{R_0};\R^N)$ with
  $$\|f_n\|_{L^2(B_{R_0})}+ \|F_n\|_{L^2(B_{R_0})} =1,$$
  such that, calling $w_n$ the corresponding solution to \eqref{pb1}, we have
\begin{equation}\label{new10genn}
    \|w_{n}\|_{L^2(B_{s})}\geq n.
\end{equation}

By passing to subsequences thanks to Proposition~5.12 in \cite{FR23},
without loss of generality, we can assume that, as $n\to+\infty$, $\sigma_n \to \sigma $ almost everywhere in $B_{R_0}$, $\eta _n$ converges to $\eta$ with respect to weak-$\ast$ convergence in $L^\infty (B_{R_0} $), $k_n \to k \in \R$, $f_n \rightharpoonup f$ and $F_n \rightharpoonup F$ with respect to weak convergence in $L^2(B_{R_0})$ and $L^2(B_{R_0},\R^N)$, respectively.

For any $n\in\N$, we call $a_n=\|w_{n}\|_{L^2(B_{s})}$ and we consider $v_n=w_n/a_n$, thus $\|v_n\|_{L^2(B_s)}=1$ and it satisfies
\begin{equation*}
\left\{
\begin{array}{ll}\displaystyle
\mathrm{div}(\sigma _n\nabla v_n) +k^2_n\eta_n v_n=\frac{f_n}{a_n} +\mathrm{div} \left(\frac{F_n}{a_n}\right) &\text{in } \R^N \\
\displaystyle \lim_{r\to +\infty } r^{\frac{N-1}{2}} \left(\frac{\partial v_n}{\partial r} - kiv_n\right) =0   & r=|x|.
\end{array}
\right.
\end{equation*}

Using the Caccioppoli inequality \eqref{disuguaglianza} in Lemma~\ref{caccioppoli}, we have that
\begin{equation}\label{newnablav}
    \|v_n\|_{W^{1,2}(B_{R_0+1/2})} \leq C,
\end{equation}
for some $C$ independent on $n$.
Moreover, arguing as in the proof of \eqref{boundaryest},
we have that
\begin{equation}\label{boundaryest2}
 \|v_n \|_{ L^{\infty}(\partial B_{R_0+1/4})} +\left\|\frac{\partial v_n}{\partial\nu} \right\|_{ L^{\infty}(\partial B_{R_0+1/4})} \leq C_1 ,
\end{equation}
for some $C_1$ independent on $n$. By \eqref{newnablav} and through \eqref{Greenrep} applied to $v_n$ with $R=R_0+1/4$ for $|x|>R_0+1/2$, we can find $C_2$, independent on $n$, such that
\begin{equation*}
    \|v_n\|_{W^{1,2}(B_s)} \leq C_2.
\end{equation*}

Passing to a subsequence, we can find $v$ such that, as $n\to+\infty$,
$v_n$ weakly converges to $v$ in $W^{1,2}(B_s)$, $v_n$ strongly converges to $v$ in $L^2(B_s)$ and in
$L^2(E)$ for any compact subset $E$ of $\R^N$, where $v$, essentially by Proposition~5.1 and Corollary~5.2 in \cite{FR23},
satisfies, if $k>0$,
\begin{equation*}
\left\{
\begin{array}{ll}\displaystyle
\mathrm{div}(\sigma \nabla v) +k^2\eta v=0 &\text{in } \R^N \\
\displaystyle \lim_{r\to +\infty } r^{\frac{N-1}{2}} \left(\frac{\partial v}{\partial r} - kiv\right) =0   & r=|x|.
\end{array}
\right.
\end{equation*}
Since $\mathrm{div}(\sigma\nabla v)+\eta v=0$ satisfies the UCP, see \cite[Proposition~5.13]{FR23}, we conclude that $v\equiv 0$ in $\R^N$. The same holds if $N\geq 3$ and $k=0$.
However, for any $n\in\N$, $\|v_n\|_{L^2(B_s)}=1$ and $v_n\to v$ as $n\to+\infty$ strongly in $L^2(B_s)$, hence $\|v\|_{L^2(B_s)}=1$ as well and this contradicts the fact that $v\equiv 0$ in $\R^N$.

Finally, continuity with respect to the wavenumber $k$ follows from \cite[Theorem~5.9]{FR23}.
\end{proof}

\section{The approximation result}\label{mainsec}
Throughout this section we keep fixed the following parameters $r>0$, $L>0$, $R_0>0$, $0<\lambda_0 < \lambda_1$, $0<\delta_0<\delta_1$ and
$ 0<k_0<k_1$. We call $I_N:=\{k\in \R:\ k\text{ satisfies }\eqref{condonk}\}$.

We also keep fixed an open set $\Omega\in \mathcal{A}(r,L,R_0)$. About $\Omega$ we assume that $\Omega$ is polyhedral and, when $N\geq 3$,  that each of its connected components is convex.
We consider also fixed the constant $s>0$ depending on $\Omega$ as in Proposition~\ref{discretizable}. As in Proposition~\ref{discretizable}, for any $h\in (0,1]$, we set
the regular triangulation $\mathcal{T}_h$ of $\overline{\Omega}$ with constants $s$ and $h$ and we set $Y^h=Y^{\mathcal{T}_h}_{\mathcal{M}^{N\times N}(\R)} \times Y^{\mathcal{T}_h}$ and $\Pi_h=\Pi_{\mathcal{T}_h}$.

Let $I_{w}\subset I_N$ which is either finite or measurable with respect to $\mathcal{H}^1$ and let $\mu_{w}=\mathcal{H}^0$ in the first case and $\mu_{w}= \mathcal{H}^1$ in the second. For simplicity, we keep fixed $I_{w}$, and the corresponding measure $\mu_{w}$, throughout this section.

Let $I_{in}\subset \mathbb{S}^{N-1}$ which is either finite or measurable with respect to $\mathcal{H}^{N-1}$ and let $\mu_{in}=\mathcal{H}^0$ in the first case and $\mu_{in}= \mathcal{H}^{N-1}$ in the second.
Let $I_{out}\subset \mathbb{S}^{N-1}$ which is either finite or measurable with respect to $\mathcal{H}^{N-1}$ and let $\mu_{out}=\mathcal{H}^0$ in the first case and $\mu_{out}= \mathcal{H}^{N-1}$ in the second.

In the sequel we consider the following a priori hypotheses on the coefficients. The condition is different if $N=2$ or $N\geq 3$, hence we introduce the following notation
\begin{equation}\label{regulardefnorm}|\sigma|_{\ast}=\left\{\begin{array}{ll}
|\sigma|_{BV(\Omega)}& \text{if }N=2
\\
|\sigma|_{W^{1,+\infty}(\Omega)} & \text{if }N\geq 3.
\end{array}
\right.
\end{equation}

\begin{defi}\label{ouradclass}
We set $X_{ad}$ as the set of $(\sigma,\eta)\in \tilde{X}$
such that
\begin{equation}\label{5.9FRbis}
    \sigma \equiv \mathbb{I}_N \quad\text{and} \quad \eta \equiv 1 \quad \text{in } \R^N \setminus \overline{\Omega}
\end{equation}
and $|\sigma|_{\ast}<+\infty$ and 
$|\eta|_{BV(B_{R_0})}<+\infty$.
 \end{defi}

\begin{rem}\label{admisremark}
 If we call the class $\mathcal{N}_{ad}$ as the set of $(\sigma,\eta,k)\in X_{ad}\times I_{w}$
such that $|\sigma|_{\ast}\leq L$ and $|\eta|_{BV(\Omega)}\leq L$, then $\mathcal{N}_{ad}$ is an admissible class in the sense of Definition~\ref{admissibleclass}. In fact, this is just a particular case of a class $\mathcal{N}$ of  \cite[Definition~5.12]{FR23}, for which we proved admissibility in Proposition~\ref{admprop}.
\end{rem}

Correspondingly, let us introduce the following \emph{regularization operator}.
\begin{defi} For any $(\sigma,\eta)\in X$, we set
\begin{equation}\label{regulardef}
R(\sigma,\eta)=\left\{\begin{array}{ll}|\sigma|_{\ast}+|\eta|_{BV(\Omega)}
& \text{if }(\sigma,\eta)\in X_{ad}
\\
+\infty & \text{otherwise}.
\end{array}
\right.
\end{equation}
\end{defi}

The target functional $F_0$ is the following. Let $(\sigma_0,\eta_0)$ be the unknown coefficients to be determined by our scattering measurements. We assume that $(\sigma_0,\eta_0)\in X_{ad}$. We set $I_{in}^{0}\subset\mathbb{S}^{N-1}$, with corresponding measure $\mu_{in}^{0}$, and $I_{out}^{0}\subset\mathbb{S}^{N-1}$, with corresponding measure $\mu_{out}^{0}$.
We call $I^{0}=I_{w}\times I_{in}^{0}\times I_{out}^{0}$ and $\mu^{0}$ the corresponding product measure
$\mu_{w}\times \mu_{in}^{0}\times\mu_{out}^{0}$ on $I^{0}$.

 Then, for a given constant $\tilde{a}>0$,
we define $F_0:X\to [0,+\infty]$ such that for any $(\sigma , \eta ) \in X$
\begin{equation}\label{targetfunc}
F_0 (\sigma , \eta )=\left\{\begin{array}{ll}
\tilde{a} R (\sigma , \eta ) & \text{if }u_{\infty}(\sigma,\eta;k,d;\hat{x})=u_{\infty}(\sigma_0,\eta_0;k,d;\hat{x})\\
& \quad\text{for any }k\in I_{w},\ d\in I_{in}^0,\ \hat{x}\in I_{out}^0\\
+\infty &\text{otherwise}.
\end{array}\right.
\end{equation}

We denote with $\varepsilon>0$ the \emph{noise level}. Associate to the noise level, we define $I_{in}^{\varepsilon}$ and $I_{out}^{\varepsilon}$, with corresponding measures $\mu_{in}^{\varepsilon}$ and $\mu_{out}^{\varepsilon}$, as the sets of incident and outgoing directions for which we collect scattering measurements.
We call $I^{\varepsilon}=I_{w}\times I_{in}^{\varepsilon}\times I_{out}^{\varepsilon}$ and $\mu^{\varepsilon}$ the corresponding product measure
$\mu_{w}\times \mu_{in}^{\varepsilon}\times\mu_{out}^{\varepsilon}$ on $I^{\varepsilon}$.
We assume that $I^{\varepsilon}\subset I^{0}$.

For any $(\sigma,\eta)\in \tilde{X}$, we define
$\mathcal{F}^{\varepsilon}(\sigma,\eta):I^{\varepsilon}\to \C$ as follows
$$\mathcal{F}^{\varepsilon}(\sigma,\eta)(k,d,\hat{x})=u_{\infty}(\sigma,\eta;k,d;\hat{x})\quad\text{for any }(k,d,\hat{x})\in I^{\varepsilon}.$$
Analogously, 
$\mathcal{F}^{0}(\sigma,\eta):I^{0}\to \C$ is defined as follows
$$\mathcal{F}^{0}(\sigma,\eta)(k,d,\hat{x})=u_{\infty}(\sigma,\eta;k,d;\hat{x})\quad\text{for any }(k,d,\hat{x})\in I^{0}.$$

We
measure $u_{\infty}(\sigma_0,\eta_0;k,d;\hat{x})$ for any $k\in I_{w}$, any $d\in I_{in}^{\varepsilon}$ and any $\hat{x}\in I_{out}^{\varepsilon}$,
that is, the exact measurements would be given by $\mathcal{F}^{\varepsilon}(\sigma_0,\eta_0)$. 
The available noisy measurements are denoted by $\mathcal{F}^{\varepsilon}_{meas}$, which we assume to be measurable with respect to $\mu^{\varepsilon}$, and the noise is defined as follows
\begin{equation}
|\mathcal{F}^{\varepsilon}_{meas}(k,d,\hat{x})-\mathcal{F}^{\varepsilon}(\sigma_0,\eta_0)(k,d,\hat{x})|\leq \varepsilon\quad\text{for any }(k,d,\hat{x})\in I^{\varepsilon}.
\end{equation}

For any $0<\varepsilon\leq 1$ we also associate $h=h(\varepsilon)\in (0,1]$ and $a=a(\varepsilon)>0$.
For simplicity we assume that $a(\varepsilon)=\tilde{a}\varepsilon^{\gamma}$ for some positive $\gamma$ to be chosen later.
We define the functional
$F_\varepsilon:X\to [0,+\infty]$ such that for any $(\sigma , \eta ) \in X$
\begin{equation}\label{approxfunc}
F_\varepsilon (\sigma , \eta )=\left\{\begin{array}{ll}
\dfrac{\|\mathcal{F}^{\varepsilon}(\sigma,\eta)-\mathcal{F}_{meas}^{\varepsilon}\|^2_{L^2(I^{\varepsilon})}}{\varepsilon^{\gamma}}+\tilde{a} R (\sigma , \eta ) & \text{if }(\sigma,\eta)\in X_{ad}\cap Y^{h(\varepsilon)}  \\
+\infty &\text{otherwise}.
\end{array}\right.
\end{equation}
Clearly, the $L^2$-norm on $I^{\varepsilon}$ is with respect to the measure $\mu^{\varepsilon}$.

We note that, through a classical application of the direct method, it is not difficult to show that $F_0$ and $F_{\varepsilon}$, for any $\varepsilon>0$, admit a minimum over $X$.

Given $F_0$,
our aim is to construct a suitable functional $F_{\varepsilon}$, for any $0<\varepsilon\leq 1$, such that
$\{F_{\varepsilon}\}_{0<\varepsilon\leq 1}$ is equicoercive, that is, there exists a compact $K\subset X$ such that
$\inf_K F_{\varepsilon}=\inf_XF_{\varepsilon}$ for any $0<\varepsilon\leq 1$, and that
$F_0 = \Gamma $-$\displaystyle \lim_{\varepsilon\to 0^+}F_{\varepsilon}$.
We need to choose carefully $I_{in}^{\varepsilon}$ and $I_{out}^{\varepsilon}$, the discretization parameter $a(\varepsilon)$, in our case the parameter $\gamma>0$, and the discretization parameter $h(\varepsilon)$. In particular, we would like both $I_{in}^{\varepsilon}$ and $I_{out}^{\varepsilon}$ to be discrete
for two important reasons.
First, this is the kind of measurements one can collect in practice and, second, with this choice the functional $F_\varepsilon$ becomes fully discrete.

We define  the \emph{discreteness parameter} of our measurements set $I^{\varepsilon}$ as
$$D(\varepsilon):=\mu_{\varepsilon}(I^{\varepsilon}).$$
Calling $m^{\varepsilon}_{in}:=\mu^{\varepsilon}_{in}(I^{\varepsilon}_{in})$ and $m^{\varepsilon}_{out}:=\mu^{\varepsilon}_{out}(I^{\varepsilon}_{out})$, 
we consider the following normalizations
 $\tilde{\mu}^{\varepsilon}_{in}:=\mu^{\varepsilon}_{in}/m^{\varepsilon}_{in}$ and
$\tilde{\mu}^{\varepsilon}_{out}:=\mu^{\varepsilon}_{out}/m^{\varepsilon}_{out}$.

We first analyze the $\Gamma$-$\liminf$ inequality, then we pass to the recovery sequence and equicoerciveness.

Let $\{\varepsilon_n\}_{n\in\N}$ be a sequence of positive numbers such that $\displaystyle \lim_n\varepsilon_n=0$.
We call $F_n=F_{\varepsilon_n}$, $I^n=I^{\varepsilon_n}$, $\mathcal{F}^n=\mathcal{F}^{\varepsilon_n}$ and so on for all the other terms involved. Let $(\sigma_n,\eta_n)\to (\sigma,\eta)$ in $X$.
Without loss of generality, we assume that $\displaystyle \liminf_nF_n(\sigma_n,\eta_n)<+\infty$. Then, by Lemma~\ref{propLip},
$$\tilde{a}R(\sigma,\eta)\leq \liminf_{n}\tilde{a}R(\sigma_n,\eta_n)\leq 
\liminf_{n}F_n(\sigma_n,\eta_n).$$
The $\Gamma$-$\liminf$ inequality follows if we show that $\mathcal{F}^0(\sigma,\eta)=\mathcal{F}^0(\sigma_0,\eta_0)$.
Without loss of generality and up to subsequences, we assume that, for some constant $C_M$,
$$F_n(\sigma_n,\eta_n)\leq C_M\quad\text{for any }n\in \N.$$
Then we have that
\begin{multline*}
\|\mathcal{F}^n(\sigma,\eta)-\mathcal{F}^n(\sigma_0,\eta_0)\|_{L^2(I^n)}\leq
\|\mathcal{F}^n(\sigma,\eta)-\mathcal{F}^n(\sigma_n,\eta_n)\|_{L^2(I^n)}\\+\|\mathcal{F}^n(\sigma_n,\eta_n)-\mathcal{F}_{meas}^n\|_{L^2(I^n)} +
\|\mathcal{F}_{meas}^n-\mathcal{F}^n(\sigma_0,\eta_0)\|_{L^2(I^n)}\\
\leq C\|(\sigma_n,\eta_n)-(\sigma,\eta)\|_X^{\beta_1}\sqrt{D(\varepsilon_n)}+\sqrt{C_M}\varepsilon_n^{\gamma/2}+\sqrt{D(\varepsilon_n)}\varepsilon_n
\end{multline*}
for some constants $C$ and $\beta_1>0$. For the estimate on the first term, we used Theorem~\ref{holdercontthm}, in particular \eqref{holderfarfield}, and Remark~\ref{admisremark}. 

More precisely,
\begin{multline*}
\int_{I^n}|u_{\infty}(\sigma,\eta;k,d;\hat{x})-u_{\infty}(\sigma_0,\eta_0;k,d;\hat{x})|^2\,d\mu_{w}(k)\,d\tilde{\mu}^n_{in}(d)\,d\tilde{\mu}^n_{out}(\hat{x})\\\leq 
3C^2\|(\sigma_n,\eta_n)-(\sigma,\eta)\|_X^{2\beta_1}\mu_{w}(I_w)+3C_M\frac{\varepsilon_n^{\gamma}}{m^n_{in}m^n_{out}}+3\mu_{w}(I_{w})\varepsilon^2_n.
\end{multline*}
We conclude that, provided $m^n_{in}\geq m$ and $m^n_{out}\geq m$ for some constant $m>0$,
\begin{equation}\label{crucialobserva}
\lim_n\int_{I^n}|u_{\infty}(\sigma,\eta;k,d;\hat{x})-u_{\infty}(\sigma_0,\eta_0;k,d;\hat{x})|^2\,d\mu_{w}(k)\,d\tilde{\mu}^n_{in}(d)\,d\tilde{\mu}^n_{out}(\hat{x})=0.
\end{equation}

The next lemma is a crucial step towards the $\Gamma$-$\liminf$ inequality.

\begin{lem}\label{Gammaliminiflemma}
Under the previous notation and assumptions, let us assume that there exists a constant $m>0$ such that, for any $\varepsilon$, $0<\varepsilon\leq 1$ we have
$m^{\varepsilon}_{in}\geq m$ and $m^{\varepsilon}_{out}\geq m$. The measures $\tilde{\mu}^{\varepsilon}_{in}$ and $\tilde{\mu}^{\varepsilon}_{out}$ are considered to be
measures on $\mathbb{S}^{N-1}$ by extending them to $0$ outside $I^{\varepsilon}_{in}$ and $I^{\varepsilon}_{out}$, respectively.

If, as $\varepsilon\to 0^+$, $\tilde{\mu}^{\varepsilon}_{in}\rightharpoonup \tilde{\mu}^0_{in}$ and 
$\tilde{\mu}^{\varepsilon}_{out}\rightharpoonup \tilde{\mu}^0_{out}$ weakly in the sense of measures on $\mathbb{S}^{N-1}$, then
$$ u_{\infty}(\sigma,\eta;k,d;\hat{x})=u_{\infty}(\sigma_0,\eta_0;k,d;\hat{x})\quad\text{for }\tilde{\mu}_0\text{-a.e. }(k,d,\hat{x})\in I_{w}\times \mathbb{S}^{N-1}\times \mathbb{S}^{N-1}$$
where $\tilde{\mu}_0=\mu_{w}\times\tilde{\mu}^0_{in}\times \tilde{\mu}^0_{out}$.
\end{lem}

\begin{proof}
By \eqref{crucialobserva} we infer that
$$\int_{I_{w}\times \mathbb{S}^{N-1}\times\mathbb{S}^{N-1}}\!\!\!\!|u_{\infty}(\sigma,\eta;k,d;\hat{x})-u_{\infty}(\sigma_0,\eta_0;k,d;\hat{x})|^2\,d\mu_{w}(k)\,d\tilde{\mu}^0_{in}(d)\,d\tilde{\mu}^0_{out}(\hat{x})=0$$
and the proof is concluded.
\end{proof}

We now turn our attention to the recovery sequence.

\begin{lem}\label{Gammalimsuplemma}
Let $(\sigma,\eta)$ be such that $R(\sigma,\eta)<+\infty$. We fix $0<\alpha<1/2$.
Then, there exist constants $C_1$ and $C_2$ such that for any $0<h\leq 1$, we can find $(\sigma_h,\eta_h)\in X_{ad}\cap Y^{h}$ such that
$$\|(\sigma_h,\eta_h)-(\sigma,\eta)\|_X\leq C_1(1+R(\sigma,\eta))h^{\alpha}\quad\text{and}\quad R(\sigma_h,\eta_h)\leq C_2(1+R(\sigma,\eta))$$
and
$$R(\sigma_h,\eta_h)\to  R(\sigma,\eta)\quad\text{as }h\to 0^+.$$
Here the constants $C_1$ and $C_2$ depends on $r$, $L$, $R$, $s$, $\lambda_0$, $\lambda_1$, $\delta_0$ and $\delta_1$ only.

Hence,
for every $(k,d,\hat{x})\in I^0$ we have
\begin{equation}\label{errorestimate}
|u_{\infty}(\sigma_h,\eta_h,k,d,\hat{x})-u_{\infty}(\sigma,\eta,k,d,\hat{x})|\leq \tilde{C}_1 \sqrt{C_1(1+R(\sigma,\eta))}h^{\alpha/2}.
\end{equation}
\end{lem}

\begin{proof} About $\eta_h$, $0<h\leq 1$, we choose the one constructed in \cite[Proposition~4.1]{FeR}. When $N=2$, again we use \cite[Proposition~4.1]{FeR}
to construct $\sigma_h$.

When $N\geq 3$, we need to modify the argument for the construction of $\sigma_h$ as follows. Let us consider the different connected components of $\Omega$. By the regularity of $\Omega$, there exists a constant $r_0>0$, depending on $r$, $L$ and $R_0$ only, such that the distance of any two different connected components of $\Omega$ is greater than or equal to $r_0$. Therefore we can argue connected component by connected component. In other words, we can assume, without loss of generality, that $\Omega$ has only one connected component. We recall that we assume that any connected component of $\Omega$ is convex. By Theorem~\ref{extensionthm} we can assume that $\sigma$ is a Lipschitz function on $\R^N$ with values in $\mathcal{M}(\lambda_0,\lambda_1)$ and that $|\sigma|_{W^{1,\infty}(\Omega)}=\mathrm{Lip}(\sigma,\Omega)=\mathrm{Lip}(\sigma,\R^N)=|\sigma|_{W^{1,\infty}(\R^N)}$.

 Let $\xi$ be a fixed positive symmetric mollifier, that is, $\xi \in C^{\infty}_0(B_1(0))$,  $\xi \geq 0$, $\int _{B_1(0)} \xi =1$ and such that $\xi (x)$ depends only on $|x|$ for any $x\in B_1(0)$. For any $\delta >0$, we call 
    \begin{equation*}
        \xi _\delta (x) = \delta ^{-N} \xi (x/\delta ), \quad x\in \R ^N,
    \end{equation*}
and, for any $\delta$, $0<\delta\leq 1$,
\begin{equation*}
        \sigma_\delta =\xi_\delta \ast \sigma,
    \end{equation*}
 where as usual $\ast$ denotes the convolution which is applied entry by entry.
We have that $\sigma_{\delta}(x)\in \mathcal{M}(\lambda_0,\lambda_1)$ for any $x\in \R^N$. Moreover, for an absolute constant $C$,
$$|\sigma_{\delta}|_{W^{1,\infty}(\R^N)}\leq |\sigma|_{W^{1,\infty}(\Omega)}\quad\text{and}\quad \|D^2 \sigma_\delta\|_{L^{\infty}(\R^N)} \leq C|\sigma|_{W^{1,\infty}(\Omega)}\delta ^{-1},$$
and
$$\|\sigma-\sigma_{\delta}\|_{L^{\infty}(\R^N)}\leq C|\sigma|_{W^{1,\infty}(\R^N)}\delta.$$

For fixed $h$, $0<h\leq 1$, we consider $\Pi_h(\sigma_{\delta})$. By Theorem~\ref{Ciarletestimteo}, in particular by \eqref{Lipestimateinter}, we have that
\begin{equation}\label{firstest}
|\Pi_h(\sigma_{\delta})|_{W^{1,\infty}(\Omega)}\leq |\sigma|_{W^{1,\infty}(\Omega)}(1+C_1h\delta^{-1}),
\end{equation}
where $C_1$ depends on $s$ only. Hence,
$$\|\sigma_{\delta}-\Pi_h(\sigma_{\delta})\|_{L^{\infty}(\Omega)}\leq |\sigma|_{W^{1,\infty}(\Omega)}(2+C_1h\delta^{-1}  )    h$$
and
$$\|\sigma-\Pi_h(\sigma_{\delta})\|_{L^{\infty}(\Omega)}\leq |\sigma|_{W^{1,\infty}(\Omega)}[(2+C_1h\delta^{-1}  )h+C\delta].$$
Let us now choose $\delta=\delta(h)=h^{1/2}$ and call $\tilde{\sigma}_h=\Pi_h(\sigma_{h^{1/2}})$. Then, for some $C_2$ depending on $s$ only,
$$\|\sigma-\tilde{\sigma}_h\|_{L^{\infty}(\Omega)}\leq C_2  |\sigma|_{W^{1,\infty}(\Omega)}h^{1/2},$$
and
$$|\tilde{\sigma}_h|_{W^{1,\infty}(\Omega)}\leq C_2|\sigma|_{W^{1,\infty}(\Omega)}.$$
Finally, by Lemma~\ref{propLip} and by \eqref{firstest}, we immediately conclude that
$$|\sigma|_{W^{1,\infty}(\Omega)}=\lim_{h\to 0^+}|\sigma_h|_{W^{1,\infty}(\Omega)}.$$

Finally, in both cases, \eqref{errorestimate} easily follows by \eqref{holderfarfield}.
\end{proof}

\begin{rem}
If $N\geq 3$ and we assume $\eta\equiv 1$, then we can improve the order of convergence of \eqref{errorestimate} from $h^{\alpha/2}$, with $0<\alpha<1/2$, to $h^{1/2}$.
\end{rem}

The following recovery sequence and equicoerciveness result 
is an immediate consequence of Lemma~\ref{Gammalimsuplemma}.
\begin{prop}\label{recoveryprop}
We fix $0<\alpha<1/2$ and we pick $h=h(\varepsilon)=\varepsilon^{2/\alpha}$ for any $0<\varepsilon\leq 1$. We assume that there
exists $\gamma>0$ such that $D(\varepsilon)\varepsilon^{2-\gamma}\to 0$ as $\varepsilon\to 0^+$.

Then  for any $(\sigma,\eta)$ such that $F_0(\sigma,\eta)<+\infty$, calling $(\sigma_{\varepsilon},\eta_{\varepsilon})=(\sigma_{h(\varepsilon)},\eta_{h(\varepsilon)})$, we have that
$$\lim_{\varepsilon\to 0^+}(\sigma_{\varepsilon},\eta_{\varepsilon})=(\sigma,\eta) \quad\text{in }X$$
and
\begin{equation}\label{recovery}
\lim_{\varepsilon\to 0^+}F_{\varepsilon}(\sigma_{\varepsilon},\eta_{\varepsilon})=F_0(\sigma,\eta).
\end{equation}

Moreover, there exists $K$, a  compact subset of $X$, such that $\inf_{X}F_{\varepsilon}=\inf_KF_{\varepsilon}$.
\end{prop}

\begin{proof}
We fix $0<\alpha<1/2$. Then for any $\varepsilon>0$ and $0<h=h(\varepsilon)\leq 1$, we have
\begin{multline}
|\mathcal{F}^{\varepsilon}(\sigma_{h(\varepsilon)},\eta_{h(\varepsilon)})(k,d,\hat{x})-  \mathcal{F}^{\varepsilon}_{meas}(k,d,\hat{x})|\\\leq \tilde{C}_1 \sqrt{C_1(1+R(\sigma,\eta))}h(\varepsilon)^{\alpha/2}+ \varepsilon\quad\text{for any }(k,d,\hat{x})\in I^{\varepsilon}.
\end{multline}
Hence, it is enough to choose $(\sigma_{\varepsilon},\eta_{\varepsilon})=(\sigma_{h(\varepsilon)},\eta_{h(\varepsilon)})$
with $h(\varepsilon)=\varepsilon^{2/\alpha}$
 to conclude that
\eqref{recovery} holds.

We apply the previous procedure to $(\sigma,\eta)=(\sigma_0,\eta_0)$.
We obtain that, for some constant $C$,
$$F_{\varepsilon}((\sigma_{h(\varepsilon)},\eta_{h(\varepsilon)})\leq C\quad\text{for any }0<\varepsilon\leq 1.$$
If we call $K=\{(\sigma,\eta)\in X:\ \tilde{a}R(\sigma,\eta)\leq C\}$, the equicoercivity property holds.
\end{proof}

We are ready to state our convergence result.
We begin with the following.
\begin{defi}\label{admmeasdef}
 We say that
$I^0$ is \emph{admissible} if 
$I_{w}$ is either finite or an interval, $I^0_{in}$ is finite or $\mathbb{S}^{N-1}$ and $I^{0}_{out}$ is finite or $\mathbb{S}^{N-1}$.

When $I^0$ is admissible, we say that $I^{\varepsilon}$, for $0<\varepsilon\leq 1$, is \emph{admissible} if the following holds.
\begin{enumerate}[(a)]
\item There exists a constant $m>0$ such that, for any $\varepsilon$, $0<\varepsilon\leq 1$ we have
$m^{\varepsilon}_{in}\geq m$ and $m^{\varepsilon}_{out}\geq m$.
\item\label{bcondit}
 There
exists $\gamma>0$ such that $D(\varepsilon)\varepsilon^{2-\gamma}\to 0$ as $\varepsilon\to 0^+$.
\item
Whenever $I^0_{in}$ is finite we define $I^{\varepsilon}_{in}=I^0_{in}$ for any $0<\varepsilon\leq 1$.
Otherwise, we require that, as $\varepsilon\to 0^+$, $\tilde{\mu}^{\varepsilon}_{in}\rightharpoonup \tilde{\mu}^0_{in}$
weakly in the sense of measures where 
there exist an open subset of $\mathbb{S}^{N-1}$, $S_{in}$, and a positive constant $c_{in}$ such that for any compact $K_1\subset S_{in}$ we have
$\tilde{\mu}^0_{in}(K_1)\geq c_{in}\mathcal{H}^{N-1}(K_1)$.
\item Whenever $I^0_{out}$ is finite we define $I^{\varepsilon}_{out}=I^0_{out}$ for any $0<\varepsilon\leq 1$.
Otherwise, we require that, as $\varepsilon\to 0^+$, $\tilde{\mu}^{\varepsilon}_{out}\rightharpoonup \tilde{\mu}^0_{out}$
weakly in the sense of measures where 
there exist an open subset of $\mathbb{S}^{N-1}$, $S_{out}$, and a positive constant $c_{out}$ such that for any compact $K_1\subset S_{out}$ we have
$\tilde{\mu}^0_{out}(K_1)\geq c_{out}\mathcal{H}^{N-1}(K_1)$.
\end{enumerate}
\end{defi}

\begin{theorem}\label{mainthm}
Let $I^0$ and $I^{\varepsilon}$, $0<\varepsilon\leq 1$, be admissible in the sense of Definition~\ref{admmeasdef}. Let $\gamma>0$ be as in \textnormal{(\ref{bcondit})} of
Definition~\ref{admmeasdef}. We fix $0<\alpha<1/2$ and we pick $h=h(\varepsilon)=\varepsilon^{2/\alpha}$ for any $0<\varepsilon\leq 1$.

Then $F_0=\Gamma$-$\displaystyle \lim_{\varepsilon\to 0^+}F_{\varepsilon}$ and the family $\{F_{\varepsilon}\}_{0<\varepsilon\leq 1}$ is equicoercive.

Moreover, if the family $\{(\sigma_{\varepsilon},\eta_{\varepsilon})\}_{0<\varepsilon\leq 1}$ satisfies
$\displaystyle
\lim_{\varepsilon\to 0^+}F_{\varepsilon}(\sigma_{\varepsilon},\eta_{\varepsilon})=\lim_{\varepsilon\to 0^+}\inf_{X}F_{\varepsilon}$,
we have that $\{(\sigma_{\varepsilon},\eta_{\varepsilon})\}_{0<\varepsilon\leq 1}$ is precompact in $X$ and
$$\displaystyle \lim_{\varepsilon\to 0^+}\mathrm{dist}\left(  (\sigma_{\varepsilon},\eta_{\varepsilon}),\argmin(F_0)\right)=0,$$
where the distance is with respect to the norm in $X$.
\end{theorem}

\begin{proof} Lemma~\ref{Gammaliminiflemma} provides the $\Gamma$-$\liminf$ inequality. To overcome the fact that we have equality only $\tilde{\mu}_0$-a.e., we use continuity with respect to $k$, thanks to Propostion~\ref{admprop}, and the analyticity of the far-field with respect both to $d$ and $\hat{x}$.
The existence of a recovery sequence and equicoerciveness follow from Proposition~\ref{recoveryprop}.
The conclusion is an immediate application of the Fundamental Theorem of $\Gamma$-convergence, Theorem~\ref{fundGammathm}.
\end{proof}

We conclude the paper with an important example of admissibility in the sense of Definition~\ref{admmeasdef} and a remark on the convexity assumption, when $N\geq 3$, of our penetrable obstacles.

\begin{exam}\label{dataexample}
 In this example we show how to construct an admissible family of measurements sets when $I^0_{in}$ is $\mathbb{S}^{N-1}$ and $I^{\varepsilon}_{in}$ is finite. The same construction applies to the case when $I^0_{out}$ is $\mathbb{S}^{N-1}$ and $I^{\varepsilon}_{out}$ is finite

Let us fix a parameter $\delta>0$. Let $S$ be an open subset of $\mathbb{S}^{N-1}$ and let us consider a finite number of points $W(\delta)=\{\omega_i\}_{i=1}^{\ell(\delta)}\subset S$
such that $S\subset \bigcup_{i=1}^{\ell(\delta)}B_{\delta}(\omega_i)$. We can construct $W(\delta)$ in such a way that its number of points $\ell(\delta)$ is
bounded by $C\delta^{-(N-1)}$, $C$ being an absolute constant. Let $\mu_{\delta}=\mathcal{H}^0|_{W(\delta)}$ 
and $\tilde{\mu}_{\delta}=\mu_{\delta}/\ell(\delta)$.

Let $\{\delta_n\}_{n\in\N}$ be a sequence of positive numbers converging to $0$. By compactness, we can assume that $\tilde{\mu}_{\delta_n}$ weakly converges in the sense of measures to
$\tilde{\mu}_0$ as $n\to +\infty$. Hence for $K_1$ compact subset of $S$,
\begin{multline*}
\mathcal{H}^{N-1}(K_1)\leq \mathcal{H}^{N-1}\left(\bigcup_{\omega_i\in K_1}B_{\delta}(\omega_i)\cap\mathbb{S}^{N-1}\right)
\\\leq C_0\delta^{N-1}\ell(\delta_n)\tilde{\mu}_{\delta_n}(K_1)\leq C_0C\tilde{\mu}_{\delta_n}(K_1),
\end{multline*}
$C_0$ being an absolute constant as well, therefore
$$\mathcal{H}^{N-1}(K_1)\leq C_0C\tilde{\mu}_0(K_1).$$

Finally, we assume that both $I^{\varepsilon}_{in}$ and $I^{\varepsilon}_{in}$ are finite,
with number of points given by $\ell(\delta_{in})$ and $\ell(\delta_{out})$, and that, for some absolute constant $C$,
$\ell(\delta_{in})\leq C\delta_{in}^{-(N-1)}$ and $\ell(\delta_{out})\leq C\delta_{out}^{-(N-1)}$
hold.
In order to guarantee that condition \textnormal{(\ref{bcondit})} of
Definition~\ref{admmeasdef} holds, it is enough to choose $\delta_{in}(\varepsilon)$ and $\delta_{out}(\varepsilon)$ so that
$(\delta_{in}(\varepsilon)\delta_{out}(\varepsilon))^{N-1}$ is of the order of $\varepsilon^{\beta}$ for some $0<\beta<2$, or in other words the maximum
number of measurements allowed when the noise level is $\varepsilon$ is about $\varepsilon^{-\beta}$ for some $0<\beta<2$.
\end{exam}

\begin{rem}\label{finalremark}
Let $\{\varepsilon_n\}_{n\in\N}$ be a sequence of positive numbers converging to $0$.
Under the assumptions of Theorem~\ref{mainthm}, for any $n\in\N$, let $(\sigma_n,\eta_n)$ be such that, for some constant $C$, we have
$$F_{\varepsilon_n}(\sigma_n,\eta_n)\leq C.$$
Then the sequence $\{(\sigma_n,\eta_n)\}_{n\in\N}$ is precompact in $X$ and 
$$\displaystyle \lim_{n}\mathrm{dist}\left(  (\sigma_n,\eta_n), S_C\right)=0,$$
where $S_C=\{(\sigma,\eta):\ F_0(\sigma,\eta)\leq C\}$ and
again the distance is with respect to the norm in $X$. That is, up to subsequences, $(\sigma_n,\eta_n)$ converges to $(\sigma_{\infty},\eta_{\infty})$ such that
$\mathcal{F}^0(\sigma_{\infty},\eta_{\infty})=\mathcal{F}^0(\sigma_0,\eta_0)$ and $R(\sigma_{\infty},\eta_{\infty})$ is finite. However we can no longer guarantee that 
$R(\sigma_{\infty},\eta_{\infty})$ is minimal with respect to $R(\sigma,\eta)$ for all $(\sigma,\eta)$ satisfying $\mathcal{F}^0(\sigma,\eta)=\mathcal{F}^0(\sigma_0,\eta_0)$.

In many practical cases, this result, although weaker, might be enough. The important remark is that this weaker result holds, by straightforward modifications of the proofs,
even when we drop the assumption that the connected components of $\Omega$ are convex.
\end{rem}

\bigskip

\noindent
\textbf{Acknowledgements.}
DDD and LR are supported by the Italian MUR through the PRIN 2022 project “Inverse problems in PDE: theoretical and numerical analysis”, project code: 2022B32J5C, CUP F53D23002710006, under the National Recovery and Resilience Plan (PNRR), Italy, Mission 04 Component 2 Investment 1.1 funded by the European Commission - NextGeneration EU programme. DDD is a member of the INdAM research group GNAMPA. LR is also supported by the INdAM research group GNAMPA through 2025 and 2026 projects.


\begin{thebibliography}{99}
\bibitem{Aleetal}
G.~Alessandrini, L.~Rondi, E.~Rosset and S.~Vessella,
\emph{The stability for the Cauchy problem for elliptic equations}, Inverse Problems \textbf{25} (2009) 123004 (47pp).

\bibitem{B02} A.~Braides, \emph{$\Gamma$-convergence for Beginners}, Oxford University Press, Oxford, 2002.

\bibitem{Cak-Vog}F.~Cakoni and M.~S.~Vogelius,
\emph{Transmission Eigenvalues and Non-scattering},
preprint arXiv:2602.06250 (2026).

\bibitem{C78} P.~G.~Ciarlet, \emph{The Finite Element Method for Elliptic Problems}, North-Holland, Amsterdam, 1978.
\bibitem{CK98} D.~Colton and R.~Kress, \emph{Inverse Acoustic and Electromagnetic Scattering Theory}, Springer-Verlag, Berlin Heidelberg New York, 1998.
\bibitem{DM93}  G.~Dal Maso, \emph{An Introduction to $\Gamma$-convergence}, Birkh\"auser, Boston, 1993.

\bibitem{Edel1} 
H.~Edelsbrunner and D.~R.~Grayson, \emph{Edgewise subdivision of a simplex}, in
Proceedings of the Fifteenth Annual Symposium on Computational Geometry
(Miami Beach, FL, 1999 ), ACM, New York, 1999, pp. 24--30.
\bibitem{Edel2} 
H.~Edelsbrunner and D.~R.~Grayson, \emph{Edgewise subdivision of a simplex},
Discrete Comput. Geom. \textbf{24} (2000) 707--719.
\bibitem{EHN96} H.~W.~Engl, M.~Hanke and A.~Neubauer, \emph{Regularization of Inverse Problems}, Kluwer Academic
Publishers, Dordrecht Boston London, 1996.
\bibitem{FeR} A.~Felisi and L.~Rondi, \emph{Full discretization and regularization for the Calder\'on problem}, J. Differential Equations \textbf{410} (2024) 513--577.
\bibitem{FR23} M.~Fornoni and L.~Rondi, \emph{Mosco convergence of Sobolev spaces and Sobolev inequalities for nonsmooth domains}, 
 Calc. Var. Partial Differential Equations \textbf{62} (2023) 15 (30pp).
 \bibitem{GT01} D.~Gilbarg and N.~S.~Trudinger, \emph{Elliptic Partial Differential Equations of Second Order}, Springer-Verlag, Berlin Heidelberg New York, 1998.

\bibitem{Isak}
V.~Isakov,
\emph{On uniqueness in the inverse transmission scattering problem},
Comm. Partial Differential Equations \textbf{15} (1990) 1565--1587.

\bibitem{Isak06}
V.~Isakov,
\emph{Inverse Problems for Partial Differential Equations},
Springer, New York, 2006.

\bibitem{Kir} M.~D.~Kirszbraun,
\emph{\"Uber die zusammenziehende und Lipschitzsche Transformationen},
Fund. Math. \textbf{22} (1934) 77--108. 


\bibitem{Leb} N.~N.~Lebedev, \emph{Special Functions and their Applications}, Prentice-Hall, Englewood Cliffs NJ, 1965.
\bibitem{LRX19} H.~Liu, L.~Rondi and J.~Xiao, \emph{Mosco convergence for $H(\mathrm{curl})$ spaces, higher integrability for Maxwell's equations, and stability in direct and inverse EM scattering problems},  J. Eur. Math. Soc. (JEMS) \textbf{21} (2019) 2945--2993.
\bibitem{MS34} E.~J.~McShane, \emph{Extension of range of functions}, Bull. Amer. Math. Soc. \textbf{40} (1934) 837--842.
\bibitem{Men-Ron} G.~Menegatti and L.~Rondi, \emph{Stability for the acoustic scattering problem for sound-
hard scatterers}, Inverse Probl. Imaging \textbf{7} (2013) 1307--1329.
\bibitem{M63} N. G. Meyers, \emph{An $L^p$-estimate for the gradient of solutions of second order
elliptic divergence equations}, Ann. Scuola Norm. Sup. Pisa (3) \textbf{17} (1963) 189--206.

\bibitem{Riv-Bar-Ob}
C.~Rivas, P.~Barbone and A.~Oberai,
\emph{Divergence of finite element formulations for inverse problems treated as optimization problems},
 J. Phys.: Conf. Ser. \textbf{135} (6th International Conference on Inverse Problems in Engineering: Theory and Practice) (2008) 012088 (8 pp).

\bibitem{Ron08}
L.~Rondi,
\emph{On the regularization of the inverse conductivity problem with discontinuous conductivities}, Inverse Probl. Imaging \textbf{2} (2008) 397--409.


\bibitem{R16} L.~Rondi, \emph{Discrete approximation and regularisation for the inverse conductivity problem}, Rend. Istit. Mat. Univ. Trieste \textbf{48} (2016) 315--352.
\end{thebibliography}
\end{document}